\pdfoutput=1
\documentclass{amsart}

\usepackage{amsmath}
\usepackage{amssymb}
\usepackage{amsthm}
\usepackage{mathtools}
\usepackage{verbatim}

\usepackage{tikz}
\usepackage{subcaption}
\usetikzlibrary{arrows} 
\usetikzlibrary{hobby}
\usetikzlibrary{fadings}
\usetikzlibrary{positioning}
\tikzset{
    base/.style = {rectangle, rounded corners, draw=black, minimum width=4cm, minimum height=1cm, text centered, font=\sffamily},
    box/.style = {rectangle, rounded corners, minimum width=5cm, minimum height=2cm, text centered, draw=gray, fill=black!15},
    header/.style={
    label={[rectangle, fill=white, draw, anchor=center, minimum width=2cm, node font=\ttfamily, name=\tikzlastnode-header]north:{#1}}}}
\tikzstyle{arrow} = [very thick, ->, >=stealth]

\usepackage{shortcuts}

\begin{document}

\title{Collet-Eckmann type conditions and conformal welding of unicritical quadratic laminations}
\author{Linhang Huang}
\address{Department of Mathematics, University of Washington, Seattle, WA 98203}
\email{lhhuang@uw.edu}
\pagestyle{plain}

\begin{abstract}
    In this paper, we introduce a Collet-Eckmann type condition for the unicritical laminations on the unit circle. We prove that this condition implies the lamination admits a Hölder continuous conformal welding which produces a Julia set for some unicritical polynomial. In consequence, we present a new proof that almost all angles on the unit circle produce quadratic polynomials with Hölder Fatou components, without the use of Beurling’s theorem.
\end{abstract}

\maketitle

\section{Introduction}
\subsection{Motivation \& Background}
We investigate the geometry of Julia sets and their corresponding polynomials. We focus on the cases where polynomials are unicritical ($p(z)=z^d+c$) and Julia sets are topologically "trees" (locally connected unique geodesic spaces, or \textit{dendrites}). In \cite{collet1983positive}, the authors introduced what is now called the \textit{Collett-Eckmann (CE) condition} for rational maps. In \cite{graczyk1998collet}, J. Graczyk and S. Smirnov proved that CE condition is sufficient for Fatou components of the polynomials to be Hölder domains. The converse of this statement is proven in \cite{przytycki2000holder}. In particular, this implies that the Julia sets for such polynomials are (quasi-)conformally removable \cite{jones2000removability}.  S. Smirnov \cite{smirnov2000symbolic} has also shown that with the exception of a set of zero harmonic measure, all the points on the boundary of the Mandelbrot set correspond to CE polynomials. This shows that CE is a generic condition for unicritical polynomials and implies nice properties for their Julia sets. The CE condition has been extensively studied for rational maps. One might refer to \cite{aspenberg2013collect,aspenberg2024slowly,astorg2019collet, bernard1994dynamique, blokh1998collet, dujardin2008distribution, gao2014summability, graczyk2000harmonic, mihalache2007collet, przytycki1995measure, przytycki1998iterations, przytycki1999rigidity} and references therein.

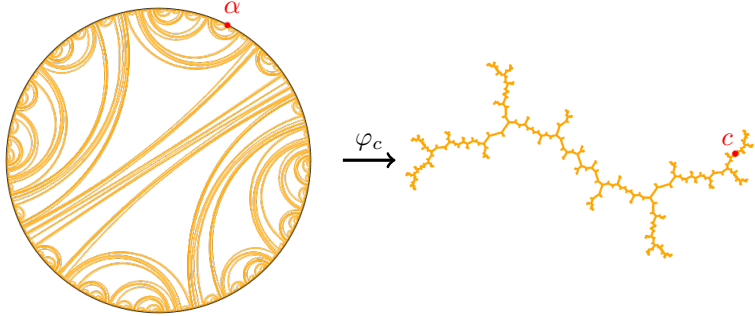
\begin{figure}[ht]
    \begin{tikzpicture}[scale=0.7]
        \node[inner sep=0pt] (julia) at (4,0){\includegraphics[scale=1.4]{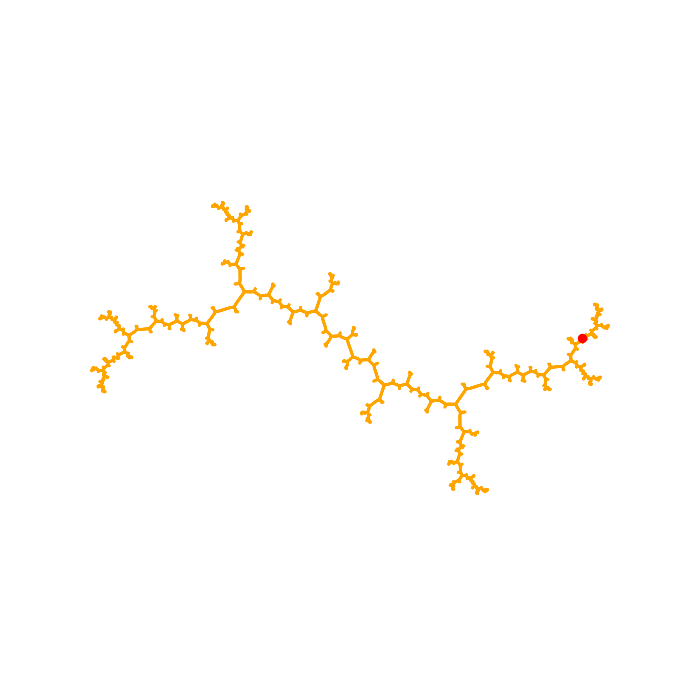}};
        \node at (-2.6,2.9) {\textcolor{red}{$\alpha$}};
        \node[inner sep=0pt] (lami) at (-4,0){\includegraphics[scale=0.3]{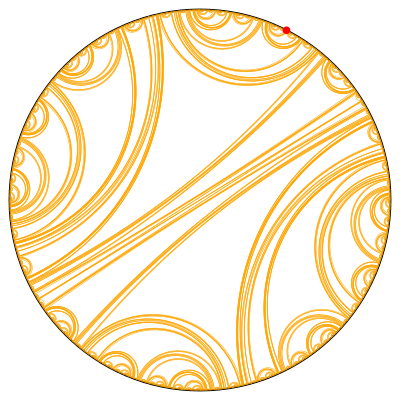}};
        \draw[->,very thick] (-0.5,0) -- (0.5,0) node[pos=0.5, above] {$\varphi_c$};
        \node at (6.8,0.4) {\textcolor{red}{$c$}};
    \end{tikzpicture}
    \vspace{-20pt}
    \caption{Quadratic lamination $\approx^{(2)}_\alpha$ with $\alpha \approx e^{0.35\pi i}$ and its corresponding quadratic Julia set $J_c$ with $c\approx-0.326+0.102i$. Riemann map $\varphi_c$ maps $\alpha$ to $c$ and "glues" $\approx_\alpha^{(2)}$.}
    \label{fig:example}
\end{figure}

There is a topological invariant analogue of CE called the \textit{topological Collett-Eckmann (TCE) condition}, constructed in \cite{nowicki1998topological} and \cite{przytycki1996porosity}. For unicritical polynomials, CE and TCE are equivalent (see for example, \cite{przytycki2003equivalence}). However, $TCE$ is generally a stronger condition for rational maps. It has been also extensively studied in the context of dynamical systems acted by hyperbolic rational functions \cite{bylund2022equivalence, ji2023non}. In \cite{przytycki2007statistical}, the authors show that every rational map satisfying TCE possesses a unique invariant measure carrying certain statistical properties of the map. In \cite{comman2011large}, it is shown that TCE rational maps satisfy some large-deviation principles.

Using the Riemann mapping, one can translate the geometry information of a Julia set into a flat equivalence relation (or a \textit{lamination}) on the unit circle $\T$. This was first studied by B. Thurston in the mid 1980's, whose manuscript was unpublished until included in \cite{schleicher2009complex} (see also \cite{bhattacharya2021unicritical}). Compared to the intricate geometric features of the Julia set, the corresponding lamination can also be described easily in terms of symbolic dynamics, and is parametrized by points on $\T$ and degree of the unicritical polynomial (see the example in Figure \ref{fig:example}). A. Douady and J. Hubbard \cite{douady1984etude} first discovered that polynomial Julia sets can be encored symbolically using what are now called the \textit{Hubbard trees}. Since then, a lot of progress has been made studying polynomial Julia sets using symbolic dynamics. One might refer to \cite{bielefeld1992classification,  blokh2012density, blokh2013cubic, blokh1999growing,blokh2013laminations, blokh2017combinatorial,blokh2020laminational,  childers2007wandering, kaffl2006hubbard,keller1994symbolic, kiwi2002wandering, mcmullen1994complex, poirier1993post, zeng2020criterion}, and references therein.

\subsection{Statement of Results}
Given a degree $d\geq 2$, we denote the aforementioned one-parameter family of laminations by $\{\approx_\alpha^{(d)}\}_{\alpha\in\T}$ (or just $\{\approx_{\alpha}\}_{\alpha\in\T}$ when the degree is clear in the context). The main goal of this paper is to find a "generic" condition analogous to CE on  $\{\approx_\alpha^{(d)}\}_{\alpha\in\T}$. We want such a condition to guarantee the conformal welding reconstruction and ensure that the reconstructed Julia sets are \textit{Hölder trees} (dendrites whose complements are Hölder domains).  We draw inspiration from \cite{lin2019conformal} and define the \textit{combinatorial Collett-Eckmann (CCE) condition} on $\{\approx_\alpha^{(d)}\}_{\alpha\in\T}$, which is sufficient for reconstructing Hölder tree Julia sets. 

\begin{theorem}\label{main_2}
    If $\alpha \in \T$ satisfies the combinatorial Collet-Eckmann (CCE) Condition for degree $2$, then the lamination $\approx^{(2)}_\alpha$ admits a conformal embedding, which reconstructs the corresponding Hölder tree Julia set, along with its unicritical polynomial $p_c(z)= z^2+c$.
\end{theorem}

The author does not see any obstruction against generalizing this theorem to higher degrees. We will only prove the case of degree 2 as the notations are much easier to work with. We shall also show that for the degree $2$ case, CCE condition is satisfied by generic points on the unit circle.

\begin{theorem}\label{main_theo}
    Almost all $\alpha \in\T$ satisfies the CCE condition for degree $2$. Thus $\approx_\alpha$ admits a H\"older continuous conformal welding, which reconstructs the corresponding Hölder tree Julia set, along with its unicritical polynomial $p_c(z)= z^2+c$.
\end{theorem}

In Section~\ref{Pre_sec}, we will go through the preliminaries needed to prove the results and also give the overview of the proofs. We will introduce the definition of the CCE condition and prove Theorem~\ref{main_2} in Section~\ref{CCE_sec}. In Section~\ref{SR_sec}, we will discuss the \textit{Strongly Recurrent (SR)} condition devised by S. Smirnov in \cite{smirnov2000symbolic}, which is associated with the \textit{semi-Combinatorial Collet-Eckmann (semi-CCE)} condition which we will define. In order to combine CCE and semi-CCE to prove Theorem~\ref{main_theo}, we will introduce the notion of \textit{Generalized Cylinder Sets} in Section~\ref{GCS_sec}. With them, we will prove Theorem~\ref{main_theo} in Section~\ref{theo_sec}. One might find the flowchart in Figure~\ref{fig:flow_chart} helpful in understanding the layout of the paper.

\begin{figure}[ht]
    \centering
    \scalebox{0.85}{
    \begin{tikzpicture}
        \node (non-PP) [base] {$\alpha$ not pre-periodic};
        \node (non-SR) [base, right = of non-PP] {$\alpha$ not strongly recurrent};
        \node (non-WPP) [base, right = of non-SR] {$\alpha$ not weakly pre-periodic};
        \node (semi-CCE) [box, header = semi-CCE, below= 4cm of non-SR] {\shortstack{$\forall$ covering: digit-fixing \\$\Rightarrow$ low encounter numbers}};
        \node (NG) [base, left = of semi-CCE] {\shortstack{generalized cylinder sets \\satisfy nice gluing property}};
        \node (DF) [base, below right = 2cm and 4.5cm of non-PP] {\shortstack{generalized cylinder sets \\are digit-fixing}};
        \node (CCE) [box, header = CCE, below= 2cm of semi-CCE] {\shortstack{$\exists$ covering: nice gluing \\\& low encounter number}};
        \node (hölder) [base, below = of CCE] {Hölder Weldability};

        \draw [arrow] (non-WPP) --++ (0,1) -| (non-PP);
        \draw [arrow] (non-WPP.270) --++ (0,-1.3) -| (DF);
        \draw [arrow] (non-SR.290) --++ (0,-1.3) -| (DF);
        \draw [arrow] (non-SR.250) --++ (0,-1.3) --++ (-2.39,0) --++ (0,-1.2) -| (semi-CCE-header);
        \draw [arrow] (non-PP.270) --++ (0,-1.3) --++ (2.45,0) --++ (0,-1.2) -|(NG);
        \draw [arrow] (NG) --++ (0,-2) -| (CCE-header);
        \draw [arrow] (DF) --++ (0,-4.5) -| (CCE-header);
        \draw [arrow] (semi-CCE) -- (CCE-header);
        \draw [arrow] (CCE) -- node [right,midway] {Theorem~\ref{main_2}}(hölder);

        \node  [above = 2cm of NG]{Proposition~\ref{good_scale}};
        \node  [above = 1.55cm of semi-CCE]{Theorem~\ref{smirn}};
        \node  [above = 0.65cm of DF]{Proposition~\ref{digit_matching}};
        
    \end{tikzpicture}}
    \caption{Outline for the proof of Theorem~\ref{main_theo}}
    \label{fig:flow_chart}
\end{figure}
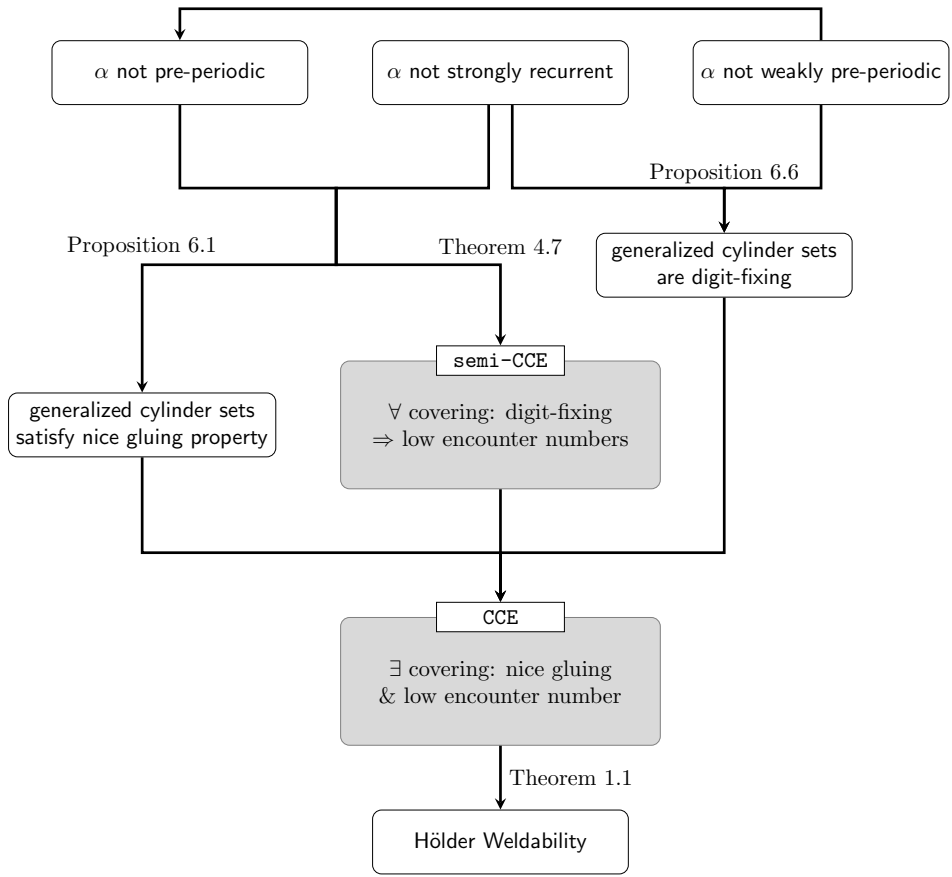

\textbf{Acknowledgments.} The author would like to thank Steffen Rohde for suggesting this project. The author is also grateful to Victor Medina and Isaiah Siegl for inspiring discussions, to Peter Lin for offering helpful references, and to Donald Marhshall for providing the codes to generate the Julia set in Figure~\ref{fig:example}.

\section{Preliminaries} \label{Pre_sec}
We denote the Riemann sphere as $\hat \C $ and the unit circle as $\T$. For a Borel set $U\subseteq \T$, we denote its $\mathcal{H}^1$-measure as $|U|$. Given an alphabet $\mathcal{A}$, we denote the set of all finite words of this alphabet $\mathcal{A}^*$ and the set of all infinite words $\mathcal{A}^\infty$. If $w$ is a word in $\mathcal{A}^*\cup \mathcal{A}^\infty$, we denote its $n$-th letter by $w[n]\in \mathcal{A}$ and with some abuse of notation, we denote its length by $|w|\in \N\cup\set{\infty}$. Given a word $w$ and an interval $I\subseteq (0,\infty)$, we denote the subword with letters on the indices $I\cap \N$ as $w|_I$. Given $a,b\in\T$, we denote the open proper arc between the two by $a\frown b$. In the case that $a = -b$, we let $a\frown b$ be the arc going from $a$ to $b$ counterclockwise.

\subsection{Laminations on $\T$ and Julia sets}\label{lamin}
Given a unicritial polynomial $p_c(z)=z^d+c$ and its corresponding Julia set $J_c$, if $J_c$ is a dendrite, by the Riemann mapping theorem, we can find a unique Riemann map $\varphi_c: \hat \C\backslash \overline \D \to \hat \C\backslash J_c$ such that $\varphi_c(\infty)=\infty$ and $\varphi_c(z) = z + O(1/z)$ near $\infty$ (we call it the \textit{hydrodynamically normalized} Riemann map). Note that $\varphi_c^{-1}\circ p_c\circ\varphi_c$ is a $d$-to-$1$ map of $\hat \C \backslash \overline{\D}$ that is $z^d + o(z^d)$ near $\infty$. It follows that \begin{equation*}
    \varphi_c^{-1}\circ p_c\circ\varphi_c(z) = z^d.
\end{equation*}
Since $J_c$ is locally connected, by the Carathéodory-Torhorst theorem \cite{torhorst1921uber}, the map $\varphi_c$ has a unique continuous extension which is surjective from $\hat \C \backslash \D$ onto $\hat \C$.  The extended $\varphi_c$ defines a equivalence relation $\equiv_c$ on $\T$ by \begin{equation*}
    x\equiv_c y \quad \Leftrightarrow \quad \varphi_c(x) = \varphi_c(y).
\end{equation*}

One can check that $\equiv_c$ is \textit{forward-invariant} in the sense that $x^d\equiv_c y^d$ whenever $x\equiv_c y$. Now given an angle $\alpha \in \T$, we set $\star_{d+1}=\star_1,\dots, \star_d$ be the $d$ roots of $z^d=\alpha$ ordered counterclockwise from $1$. Note that $z\mapsto z^d$ maps each $\star_i \frown \star_{i+1}$ bijectively to $\T \backslash\set{\alpha}$. We let $\widetilde{L_i}: \T \backslash\set{\alpha}\to \star_i \frown \star_{i+1}$ denote the inverse of such a map. We can define a new forward-invariant lamination $\sim_\alpha$ with the following:\begin{enumerate}
    \item We set $\set{\star_i}_{i=1}^d \subseteq [\star_i]_{\sim_\alpha}$.
    \item For $x\sim_\alpha y$ and $\alpha \notin \set{x,y}$, we set $\widetilde{L_i}(x)\sim_\alpha \widetilde{L_i}(y)$ for all $i=1,\dots, d$.
    \item If $x_n\to x$, $y_n\to y$ and $x_n \sim_\alpha y_n$ via (1) or (2), we set $x\sim_\alpha y$.
\end{enumerate}

We say an equivalence relation on $\T$ is a \textit{lamination} if it does not have crossings. The following lemma tells us that $\sim_\alpha$ is in fact a lamination.

\begin{lemma}[\cite{thurston2020degree}]\label{non-crossing}
    Suppose $x_1\sim_\alpha y_1$ and $x_2\sim_\alpha y_2$. If the line segment connecting $x_1$ and $y_1$ intersects that connecting $x_2$ and $y_2$, then $x_1\sim_\alpha y_1 \sim_\alpha x_2\sim_\alpha y_2$.
\end{lemma}

Additionally, $\sim_\alpha$ is a $d$-invariant unicritical lamination as in \cite{thurston2020degree}. Since $z\mapsto z^d$ conjugates with $p_c$ via $\varphi_c$, by checking the external rays, we have the following:\begin{lemma}[\cite{lin2019conformal}]
    If $J_c$ is a dendrite, for any $\alpha\in\varphi_c^{-1}(\set{c})$, we have $\sim_\alpha \subseteq \equiv_c$.
\end{lemma} 

We call $\varphi_c^{-1}(c)$ the \textit{special angles} of polynomial $p_c$.

\subsection{Symbolic Dynamics and Cylinder Sets} The lamination $\sim_\alpha$ is not easy to work with in our setting with how it is defined using limits. In this section, we will introduce an equivalent definition.
Given $\alpha \in\T$, let $\set{\star_i}\subseteq\T$ and $g_i:\T \backslash \set{\alpha}\to \star_i\frown \star_{i+1}$ be the points and functions defined above. 

\begin{definition}[Itinerary Maps, \cite{lin2018quasisymmetry}]
An \textit{itinerary map} $I^\alpha$ is a map from $\T$ to the set of infinite words with alphabet $\set{L_1,\dots, L_d,\star}$ such that for any $x\in\T$, \begin{equation*}
    I^\alpha(x)[n] := \begin{cases}
        \star,&\text{if $x^{d(n-1)} \in \set{\star_i}$},\\
        L_i,&\text{if $x^{d(n-1)} \in \star_i\frown \star_{i+1}$},
    \end{cases}\quad \text{for all $n>0$.}
\end{equation*}  
\end{definition}

In other words, $I^\alpha(x)$ tracks where $x$ jumps when iteratively applying $z\mapsto z^d$. One can check that \begin{enumerate}
    \item $I^\alpha(x)|_{[t+1,\infty)} = I^\alpha(x^{td})$.
    \item If $x\neq \alpha$, then $I^\alpha(\widetilde{L_i}(x)) = L_iI^\alpha(x)$.
    \item If $x$ is periodic under $z\mapsto z^d$, then $I^\alpha(x)$ is periodic. However, the period of $x$ can be a multiple of the period of $I^\alpha(x)$.
\end{enumerate} Therefore, we can define an equivalence relation $\approx_\alpha$ as follows \begin{equation*}
    x\approx_\alpha y \quad \Leftrightarrow \quad \text{for all $n\in\N$}, \quad I^\alpha(x)[n] = I^\alpha(y)[n] \quad \text{or} \quad \star\in \set{I^\alpha(x)[n],I^\alpha(y)[n]}.
\end{equation*}
One immediate consequence is that $\widetilde{L_i}(x) \approx_\alpha \widetilde{L_i}(y)$ if $x\approx_\alpha y$ and $\alpha \notin \set{x,y}$. That is, $\widetilde{L_i}$ preserves $\approx_\alpha$. This equivalence relation is the one in Theorems~\ref{main_2} and \ref{main_theo} and it satisfies the following:
\begin{lemma}[\cite{bandt2006symbolic, thurston2020degree}]\label{equi_lami}~\begin{enumerate}
    \item If $J_c$ is a dendrite, then for any $\alpha\in\varphi_c^{-1}(\set{c})$, we have $\equiv_c \subseteq \approx_\alpha$.
    \item We have that $\sim_\alpha \subseteq\approx_\alpha$. In the case when $\alpha$ is aperiodic under $z\mapsto z^d$, each equivalence class of $\approx_\alpha$ is finite and $\sim_\alpha=\approx_\alpha$.
\end{enumerate}
In particular, if $J_c$ is a dendrite and $c$ is aperiodic under $p_c$, then $\sim_\alpha$, $\approx_\alpha$ and $\equiv_c$ are the same lamination.
\end{lemma}
Given a finite word $w\in\set{L_1,\dots, L_d}^n$, we can then associate a map $\widetilde{w}$ such that \begin{equation*}
    \widetilde{w} = \widetilde{w[1]} \circ\widetilde{w[2]} \circ \cdots \circ \widetilde{w[n]}.
\end{equation*} Since each $\widetilde{L_i}$ is well-defined on $\T$ except at $\alpha$, $\widetilde w$ is well-definded everywhere on $\T$ except possibly at $\set{\alpha^{di}}_{i=0}^{|w|}$.

\begin{figure}[ht]
    \centering
    \begin{tikzpicture}[scale =0.85]
        \draw [line width=1pt] (0,0) circle (4);
    
        \draw [line width=1.2pt, blue] ({-4*cos(360/7)},{-4*sin(360/7)}) node [red, below left] {$\star_2$} -- ({4*cos(360/7)},{4*sin(360/7)}) node [red, above right] {$\star_1$};
        
        \harcc{0}{0}{4}{180/7}{360*9/28+180}{line width=1.2pt, blue}
        \harcc{0}{0}{4}{360*9/28}{180/7+180}{line width=1.2pt, blue}

        \harcc{0}{0}{4}{720/7}{180*9/28+360}{line width=1.2pt, blue}
        \harcc{0}{0}{4}{720/7-180}{180*9/28+180}{line width=1.2pt, blue}
        \harcc{0}{0}{4}{90/7}{180*23/28+180}{line width=1.2pt, blue}
        \harcc{0}{0}{4}{90/7+180}{180*23/28+360}{line width=1.2pt, blue}
        
        \fill [red] ({4*cos(720/7)},{4*sin(720/7)}) circle (3pt) node [red,above right=0.6pt]{$\alpha$};

        \nodeat{0}{0}{1}{120}{$L_1L_1L_1$}
        \nodeat{0}{0}{3.4}{80}{\footnotesize $L_1L_1L_2$}
        \nodeat{0}{0}{2.7}{140}{\footnotesize $L_1L_2L_1$}
        \nodeat{0}{0}{2.7}{140}{\footnotesize $L_1L_2L_1$}
        \nodeat{0}{0}{3.3}{170}{\footnotesize $L_1L_2L_2$}
        \nodeat{0}{0}{1}{300}{$L_2L_1L_1$}
        \nodeat{0}{0}{3.4}{260}{\footnotesize $L_2L_1L_2$}
        \nodeat{0}{0}{2.7}{320}{\footnotesize $L_2L_2L_1$}
        \nodeat{0}{0}{2.7}{320}{\footnotesize $L_2L_2L_1$}
        \nodeat{0}{0}{3.3}{350}{\footnotesize $L_2L_2L_2$}

    \end{tikzpicture}
    \caption{The cylinder sets for $\alpha=e^{4\pi i/7}$ associated with words of length $3$ with $d=2$.}
    \label{fig:cylinder}
\end{figure}
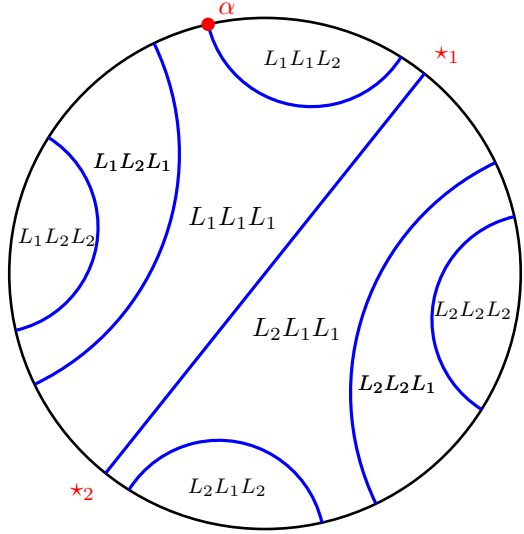

\begin{definition}[Cylinder Sets, \cite{lin2018quasisymmetry}]
    Given a finite word $w \in \set{L_1,\dots, L_d}^*$,  the \textit{cylinder set} $C(w)$ associated with $w$ is the closure of the image of $\widetilde w$ (see Figure~\ref{fig:cylinder}). Alternatively, if we define $C(L_i) = \overline{\star_i \frown \star_{i+1}}
$, then \begin{equation*}
    C(w) = \set{x\in\T: x^{d(i-1)}\in C(w[i])\quad \text{for all $i=1,\dots, n$}}.
\end{equation*}
\end{definition}
In consequence, we have that $x\in C(w)$ if $I^\alpha(x)[i] \in \set{w[i],\star}$ for all $i=1,\dots, |w|$.

\subsection{Gluing Links} In light of Lemma~\ref{equi_lami}, we would only consider aperiodic $\alpha$ such that $\sim_\alpha$ (defined dynamically) and $\approx_\alpha$ (defined with itineraries) are the same. It is clear that there are only countable periodic angles so this assumption is harmless. It is easy to see that the cylinder sets defined above are saturated closed set with respect to the quotient map defined by $\approx_\alpha$. More generally, they are \textit{gluing links}, whose definition is as follows:
\begin{definition}[Gluing links]
    A gluing link $D$ (with respect to $\approx_\alpha$) is a finite union of disjoint closed arcs $\sqcup_{i=1}^n \overline{a_1\frown b_1}$ with $a_1,b_1,\dots,a_n,b_n$ ordered counterclockwise such that $b_i \approx_\alpha a_{i+1}$ with $a_{n+1}=a_1$.
\end{definition}
In other words, a gluing link is a finite set of arcs whose end points are "linked" by $\approx_\alpha$.  With invariant properties of $\approx_\alpha$, one can easily verify the following properties of gluing links. To take pre-images more easily, we set $h(z) = z^d$ ("h" for "hopping").

\begin{proposition}[see Figure~\ref{fig:links}]\label{gluing_links}
    Given a gluing link $D\subseteq T$ with $n$ arcs, we have
    \begin{enumerate}
        \item The image $h(D)$ is either the full circle or a gluing link.
        \item If $\alpha\notin D$, then \begin{equation*}
            h^{-1}(D) = \bigsqcup_{i=1}^d \widetilde{L_i}(D),
        \end{equation*} where each $\widetilde{L_i}(D)$ is a gluing link with $n$ arcs and $|\widetilde{L_i}(D)| = |D|/d$.
        \item If $\alpha \in D$, then  \begin{equation*}
            h^{-1}(D) = \overline{\bigsqcup_{i=1}^d \widetilde{L_i}(D\backslash \set{\alpha})},
        \end{equation*} is a gluing link with $dn$ arcs and $|h^{-1}(D)| = |D|$.
    \end{enumerate}
\end{proposition}

As a consequence of Lemma~\ref{gluing_links}, given a gluing link $D$ and integer $i\geq 0$, $h^{-i}(D)$ is a disjoint union of gluing links (Here we define $h^0$ as $\id$). We call each of them an $\alpha$-\textit{component} of $h^{-i}(D)$. If $x\in h^{-i}(D)$, we denote the $\alpha$-component containing $x$ by $\comp^\alpha_x h^{-i}(D)$. To justify the notation, we observe that under the quotient topology of $\T$ given by $\approx_\alpha$, \begin{equation*}
    \comp^\alpha_x h^{-i}(D) = \comp^{ \T/_{\approx_\alpha}}_x h^{-i}(D).
\end{equation*} That is, $\alpha$-components are just topological connected components of $h^{-i}(D)$.

\begin{figure}[ht]
    \centering
        \begin{tikzpicture}[scale=0.8]
        
            \fill[teal, opacity=0.4] \arch{9}{0}{3}{30}{40} -- \arcc{9}{0}{3}{40}{50} -- \arch{9}{0}{3}{50}{205} --\arcc{9}{0}{3}{205}{210} 
                
            -- \arch{9}{0}{3}{210}{220} -- \arcc{9}{0}{3}{220}{230} -- \arch{9}{0}{3}{230}{385} --\arcc{9}{0}{3}{385}{390} -- cycle;

            \fill[teal, opacity=0.4] \arch{0}{0}{3}{60}{80} -- \arcc{0}{0}{3}{80}{100} -- \arch{0}{0}{3}{100}{410} --\arcc{0}{0}{3}{410}{420}-- cycle;

            \fill[blue, opacity=0.4] \arcc{9}{0}{3}{-170}{-165} -- \arch{9}{0}{3}{-165}{140} -- \arcc{9}{0}{3}{140}{150} -- \arch{9}{0}{3}{150}{160} -- \arcc{9}{0}{3}{160}{170} -- \arch{9}{0}{3}{170}{190} -- cycle;

            \fill[blue, opacity=0.4] \arcc{9}{0}{3}{10}{15} -- \arch{9}{0}{3}{15}{320} -- \arcc{9}{0}{3}{320}{330} -- \arch{9}{0}{3}{330}{340} -- \arcc{9}{0}{3}{340}{350} -- \arch{9}{0}{3}{350}{370} -- cycle;

            \fill[blue, opacity=0.4] \arcc{0}{0}{3}{20}{30} -- \arch{0}{0}{3}{30}{280} -- \arcc{0}{0}{3}{280}{300} -- \arch{0}{0}{3}{300}{320} -- \arcc{0}{0}{3}{320}{340} -- \arch{0}{0}{3}{340}{380} -- cycle;
                
            \draw [line width=1pt] (0,0) circle (3);
            \draw [dashed, line width=1pt, red] ({3*cos(222.5)},{3*sin(222.5)}) node [red, below left] {$\star_2$} -- ({3*cos(42.5)},{3*sin(42.5)}) node [red, above right] {$\star_1$};
            \fill [red] ({3*cos(85)},{3*sin(85)}) circle (2pt) node [above=0.6pt] {$\alpha$};
        
            \draw [->,line width=3pt] (4,0) -- (5,0) node[pos=0.4,above=1pt] {\huge $h^{-1}$};
            
            \draw [line width=1pt] (9,0) circle (3);
            \draw [dashed, line width=1pt, red] ({9+3*cos(222.5)},{3*sin(222.5)}) node [red, below left] {$\star_2$} --  ({9+0.5*cos(222.5)},{0.5*sin(222.5)}) ;
            \draw [dashed, line width=1pt, red] ({9+0.5*cos(42.5)},{0.5*sin(42.5)})  --  ({9+3*cos(42.5)},{3*sin(42.5)}) node [red, above right] {$\star_1$};
            \fill [red] ({9+3*cos(85)},{3*sin(85)}) circle (2pt) node [above=0.6pt] {$\alpha$};

            \node at ({2.3*cos(80)},{2.3*sin(80)}) {$D_1$};
            \node at ({2*cos(-30)},{2*sin(-30)}) {$D_2$};
            \node at (9,0) {$h^{-1}(D_1)$};
            \node at ({9+2.3*cos(-53)},{2.3*sin(-53)}) {$\widetilde{L_2}(D_2)$};
            \node at ({9+2.3*cos(125)},{2.3*sin(125)}) {$\widetilde{L_1}(D_2)$};
        \end{tikzpicture}
        \caption{Picture for Proposition~\ref{gluing_links} where $d=2$. The gluing link $D_1$ contains $\alpha$ and thus its preimage has one component, whereas the gluing link $D_2$'s preimage has two separate gluing links as it does not contain $\alpha$.}
        \label{fig:links}
\end{figure}
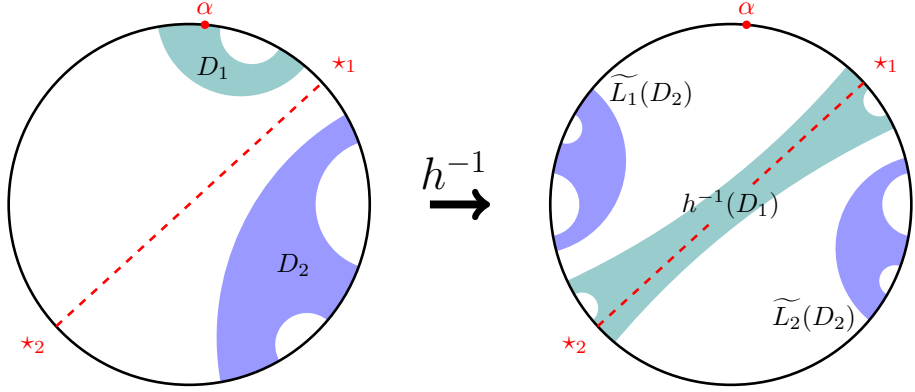

\subsection{Pulling back Gluing Links}\label{pulling_back}
In this section, we will introduce a specific way to construct families of gluing links using Proposition~\ref{gluing_links}. The proofs of the main theorems require us to extract the microscopic information of $\approx_\alpha$ at each point. In order to do that, we want to find gluing links of different scales, which we can do using the following construction: Suppose first that we have a covering of $\T$ by gluing links. We would like to denote this covering by $\set{D(x)}_{x\in\T}$, where for each point $x\in \T$, we assign a unique gluing link $D(x) \ni x$. 
\begin{lemma}\label{smaller}
    Given $x\in \T$, for all $n>0$, we have $x \in h^{-n}(D(h^n(x)))$. As a consequence, we can set \begin{equation*}
        D^{(n)}(x) := \comp_x^\alpha h^{-n}(D(h^n(x))),
    \end{equation*} and for each $n$, $\set{D^{(n)}(x)}_{x\in \T}$ forms a new covering of $\T$ by gluing links.
\end{lemma}

This statement is rather trivial as $h^n(x) \in D(h^n(x))$ by definition. Nevertheless, we would like to list it as a lemma for its usefulness. Note that by Proposition~\ref{gluing_links}, $\set{D^{(n)}(x)}$ are generally much smaller than $\set{D(x)}$. Nevertheless, we do not necessarily have $D^{(n)}(x) \subseteq D(x)$. We can think of $D^{(n)}(x)$ as arising from the sequence \begin{align}\label{seq}
    D(h^n(x)) := D^{(0)}(h^n(x))\to D^{(1)}(h^{n-1}(x))  \to \cdots \nonumber \\\cdots\to D^{(n-1)}(h^1(x)) \to D^{(n)}(x),
\end{align} 
where at each step, we take the preimage of $D^{(i)}(h^{n-i}(x))$ and pick the $\alpha$-component that contains $h^{n-i-1}(x)$. A consequence of Proposition~\ref{gluing_links} is as follows:

\begin{corollary}
    Consider the sequence of gluing links in $\eqref{seq}$. For each $i\in [1,n]\cap \N$, we have \begin{equation*}
        h(D^{(n-i+1)}(h^{i-1}(x))) = D^{(n-i)}(h^{i}(x)).
    \end{equation*} Moreover, if $D^{(n-i)}(h^i(x))$ does not contain $\alpha$, then $D^{(n-i+1)}(h^{i-1}(x))$ is a rescaled image of $D^{(n-i)}(h^{i}(x))$ under some $\widetilde{L_j}$. If $D^{(n-i)}(h^i(x))$ does contain $\alpha$, then \begin{equation*}
        D^{(n-i+1)}(h^{i-1}(x)) = h^{-1}(D^{(n-i)}(h^i(x))).
    \end{equation*}
\end{corollary}

For our purpose, we would like to think that there is a property $P$ for $\approx_\alpha$ preserved by maps $\widetilde{L_j}$. This is quite natural considering that we have $\widetilde{L_i}(x) \approx_\alpha \widetilde{L_i}(y)$ whenever $x\approx_\alpha y$. We suppose \begin{enumerate}
    \item if a gluing link $D$ with $|D| \asymp d^{-n}$ is "nice" for $P$ and $\alpha\notin D$, then $\widetilde{L_j}(D)$ is a gluing link with size $\asymp d^{-n-1}$ still nice for $P$. 
    \item If a gluing link $D$ with $|D| \asymp d^{-n}$ is nice for $P$ but $\alpha\in D$, then $h^{-1}(D)$ is a gluing link with size $\asymp d^{-n-1}$ that is "slightly worse" for $P$.
\end{enumerate}

Inspired by this, given $\set{D(x)}$, we set the \textit{encounter number} to be \begin{equation*}
    N(x,n) := \# \set{i\in [1,n]\cap \N~|~ \alpha \in D^{(n-i)}(h^i(x))}.
\end{equation*}

Intuitively, if all of $\set{D(x)}$ are nice for the property $P$ with size $\asymp 1$, then $D^{(n)}(x)$ is "acceptable" for $P$ with size about $d^{-n}$ as long as $N(x,n)$ is small.

\subsection{Structure of proofs}
In Section~\ref{CCE_sec}, we will describe the "nice gluing" property for a gluing link introduced by P. Lin and S. Rohde in \cite{lin2019conformal}. This "nice gluing" property is in fact preserved by $\widetilde{L_j}$ and thus we can use the strategy mentioned above. The CCE condition can be roughly described as \begin{enumerate}
    \item there exists a covering $\set{D(x)}$ of $\T$ by gluing links where each $D(x)$ is of comparable size and satisfies the "nice gluing property" at the macro level.
    \item For each $x\in \T$, the encounter number $N(x,n)$ is small for many $n\in \N$.
\end{enumerate}
A consequence of the second condition is that for each $x$, we will have many $D^{(n)}(x)$ of different scales that are acceptable in regards to the nice gluing property. This will imply that $\approx_\alpha$ admits a conformal welding by Theorem~1.1 of \cite{lin2019conformal}. We will provide the details and give the full proof of Theorem~\ref{main_2} in Secion~\ref{CCE_sec}.

To prove Theorem~\ref{main_theo}, we want to show that for a generic $\alpha$, we can find a covering $\set{D(x)}$ that satisfies CCE. This is rather challenging and we will dedicate Sections~\ref{SR_sec},\ref{GCS_sec} and \ref{theo_sec} to finding such a covering. In Section~\ref{SR_sec}, we will show that for a generic $\alpha$, $\approx_\alpha$ satisfies the \textit{semi-CCE} condition, which can be roughly described as \begin{quote}
    For any covering $\set{D(x)}$ of $\T$ by gluing links (with respect to $\approx_\alpha$), if each $D(x)$ satisfies the \textit{digit-fixing} condition, then for each $x\in\T$, $N(x,n)$ is small for many $n\in\N$.
\end{quote}

This digit-fixing condition (defined in Section~\ref{df}) is described using itinerary map $I^\alpha(x)$ so it is relatively trackable. With CCE and semi-CCE conditions, we are aimed to find a good candidate for $\set{D(x)}$ that has nice gluing property and satisfies the digit-fixing condition. In Section~\ref{GCS_sec}, we will discuss a special type of gluing links called \textit{generalized cylinder sets}. Like cylinder sets, the generalized cylinder sets are defined using itinerary map $I^\alpha(x)$. In consequence, we can explicitly check for which $\alpha$, the generalized cylinder sets satisfy the digit-fixing condition. To check when the generalized cylinder sets satisfy the nice gluing property, we will follow a similar strategy outlined in \cite{lin2019conformal} and use the symbolic information provided by itineraries to infer geometric nicety inside each generalized cylinder sets. In Section~\ref{theo_sec}, we will show that for a generic $\alpha$, we can find a covering of $\T$ by the generalized cylinder sets that satisfy the two properties, which implies that $\approx_\alpha$ admits a Hölder-continuous conformal welding.

\section{Combinatorial Collet-Eckmann (CCE) Condition} \label{CCE_sec}
We will give the definition of the combinatorial Collet-Eckmann condition and prove Theorem~\ref{main_2} in this section. 
\subsection{Nice gluing circuits}
\begin{figure}[ht]
    \centering
    \begin{tikzpicture}[scale=0.8]
        \draw [line width=1pt] (0,0) circle (4);
        \fill ({4*cos(120)},{4*sin(120)}) circle (3pt) node [above left=0.6pt]{$x$};
    
        \harc{0}{0}{4}{-33}{102}
        \harc{0}{0}{4}{-35}{106}
        \harc{0}{0}{4}{-37}{108}
        \harc{0}{0}{4}{-42}{112}
    
        \harc{0}{0}{4}{127}{163}
        \harc{0}{0}{4}{129}{161}
        \harc{0}{0}{4}{131}{158}
        \harc{0}{0}{4}{133}{157}
    
        \harc{0}{0}{4}{175}{307}
        \harc{0}{0}{4}{180}{305}
        \harc{0}{0}{4}{182}{303}
        
        \nodeat{0}{0}{4.4}{320}{$A_3$}
        \nodeat{0}{0}{4.4}{107}{$A'_1$}
        \nodeat{0}{0}{4.4}{130}{$A_1$}
        \nodeat{0}{0}{4.4}{160}{$A'_2$}
        \nodeat{0}{0}{4.4}{178}{$A_2$}
        \nodeat{0}{0}{4.4}{305}{$A'_3$}

        \nodeat{0}{0}{1.5}{60}{$\phi_3$}
        \nodeat{0}{0}{3.6}{145}{$\phi_1$}
        \nodeat{0}{0}{1.8}{250}{$\phi_2$}
    
        \gluepair{-45}{-30}{100}{114}{4}{0}{0}
        \gluepair{125}{134}{156}{165}{4}{0}{0}
        \gluepair{170}{184}{296}{310}{4}{0}{0}
    \end{tikzpicture}
    \caption{A nice gluing circuit around $x$}
    \label{fig:nice_gluing}
\end{figure}
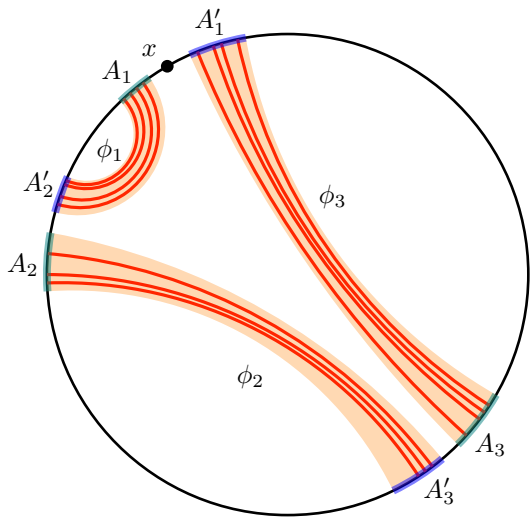

\begin{definition}[Nice gluing circuits \cite{lin2019conformal}]
    Given a lamination $\approx$ and a triple $(N,C,r)$, we define an $(N,C,r)$-\textit{nice gluing circuit} to be a set of disjoint closed arcs $A_1,A'_1,\dots, A_n,A'_n$ ordered counterclockwise in $\T$ with $A_{n+1}=A_1$,$A'_{n+1}=A'_1$ and $n\leq N$  such that \begin{enumerate}
        \item For each $j$, we have \begin{equation*}
            \diam (A_i\cup A'_i) \asymp_C r,\quad \diam A_i \asymp_C {\diam(A_i\cup A'_i)}\asymp_C \diam A'_i.
        \end{equation*}
        \item For each $j$, there exist (Borel) probability measures $\mu_j$ and $\mu'_j$ supported on $A_j$ and $A'_j$ respectively such that \begin{equation*}
            \max (\mathcal{E}(\hat \mu_j), \mathcal{E}(\hat \mu'_j)) \leq C,
        \end{equation*} where $\hat{\mu}_j$ the push-forward of $\mu_j$ under the map $x \mapsto x/\diam(A_j)$ (the same for $\hat{\mu}'_j$) and $\mathcal{E}(\mu)$ is the \textit{logarithmic energy} of $\mu$ defined as \begin{equation*}
            \mathcal{E}(\mu) = \int_{\text{supp}(\mu)^2} -\log|x-y|d\mu(x)d\mu(y).
        \end{equation*}
        We call $\mathcal{E}(\hat{\mu})$ the \textit{rescaled logarithmic energy} for $\mu$.
        
        \item For each $j$, there exists a $\mu_j$-full measure set $B_j\subseteq A_j$ and a $\mu'_{j+1}$-full measure set $B'_{j+1} \subseteq A'_{j+1}$ along with a $(\mu_j|_{B_j},\mu'_{j+1}|_{B'_{j+1}})$-measure preserving bijection $\phi_j: B_j \to B'_{j+1}$ such that \begin{equation*}
            y \approx \phi_j (y),\quad \text{for all $y\in B_j$}.
        \end{equation*}
    \end{enumerate}
\end{definition}

We say a nice gluing circuit is \textit{around} $x\in\T$ if $x$ is between some $A_i$ and $A'_i$ for some $i$ (see Figure \ref{fig:nice_gluing}). Intuitively, one can think of a nice gluing circuit as a series circuit in the context of electrical networks with the pairs $(A_j,A'_{j+1})$ acting as resistors in the circuit, and $A_j$ and $A_j'$ connected by zero-resistance paths. The parameter $N$ bounds the number of resistors in the circuit, whereas the parameter $C$ controls the resistance for each of them ($r$ describes the physical scale of the circuit). To prove Theorem~\ref{main_2}, we will use the following criterion for the existence of Hölder continuous conformal welding.

\begin{theorem}[{\cite[Theorem~1.1]{lin2018quasisymmetry}}]\label{peter_gluing}
    Given a lamination, if there are $N,C,r$ and $P > 1$ such that for every $x\in\T$, there exists an increasing sequence of integers $\set{n_j}$ with $n_j\leq Pj$ such that there exists an $(N,C,2^{-n_j}r)$-nice gluing circuit around $x$ for each $i$, then the lamination admits a Hölder continuous conformal welding.
\end{theorem}

The removability of Hölder trees then implies that the resultant dendrite of the conformal welding for $\approx_\alpha$ has to be a Julia set (One can refer Chapter 5 of \cite{lin2018quasisymmetry} for details). Note that the base $2$ in the statement can be change to any real number $d>2$ as we can replace $P$ with $\lceil \log d/\log 2\rceil P$.

\subsection{Definition of CCE Condition}

The TCE condition is a topological condition on polynomials implying that the corresponding Julia set is a Hölder tree. 
\begin{definition}[\cite{przytycki1996porosity, smirnov2000symbolic}]
    Given $c\in\C$ and consider the polynomial $p_c = z^d+c$ with corresponding Julia set $J_c$, we say $p_c$ satisfies the \textit{topological Collet-Eckmann (TCE) condition} if there exist $(r,P,M)$ such that  for any $x\in J_c$, there exists an increasing sequence of integers $\set{n_j}$ with $n_j\leq Pj$ and \begin{equation*}
            \#\set{0\leq i < n_j:c \in \comp_{p_c^{n_j-i}(x)}p_c^{-i}(B_{p_c^{n_j}(x)}(r))} \leq M.
        \end{equation*}
\end{definition}

The definition of CCE condition is inspired by the TCE condition, where gluing links play a similar role to euclidean balls intersecting the Julia set $J_c$ and the map $h$ playing the role of the polynomial $p_c$.

\begin{definition}\label{CCE}
    Given $\alpha\in\T$ and consider the lamination $\approx_\alpha$, we say $\alpha$ satisfies the \textit{combinatorial Collet-Eckmann (CCE) condition} if there exists a family of gluing links $\set{D(y)}_{y\in\T}$ with $y\in D(y)$ that satisfies the following two conditions:
    \begin{enumerate}
        \item \label{CCE1} there exist $(N,C,r)$ such that for any $x\in\T$, there exists an $(N,C,r)$-nice gluing circuit around $x$ contained in $D(x)$;
        \item \label{CCE2}there exist $(P,M)$ such that for any $x\in\T$, there exists an increasing sequence of integers $\set{n_j}$ with $n_j\leq Pj$ and \begin{equation*}
            N(x,n_j) = \#\set{0\leq i < n_j:\alpha \in \comp^\alpha_{h^{n_j-i}(x)}h^{-i}(D(h^{n_j}(x)))} \leq M.
        \end{equation*}
    \end{enumerate} 
\end{definition}

We will call the first condition of $\set{D(x)}$ the \textit{nice gluing} condition.

\subsection{Proof of Theorem~\ref{main_2}}
Following the strategy in Section~\ref{pulling_back}, we need to show that the nice gluing condition is preserved under $\widetilde{L_i}$ and will only be slightly worse when $\alpha \in D^{(n-i)}(h^i(x))$ and we need to take the preimage. Again we will only consider the case where $d=2$. Therefore we have $h(z) = z^2$. The argument should work for $d\geq 3$ with slight modification.

\begin{lemma}~\label{copying}
    Suppose that a collection $\set{D(x)}_{x\in\T}$ of gluing links satisfies Condition~\ref{CCE1} of Definition~\ref{CCE} with parameters $(N,C,r)$. Given any integer $M$, there exist $N_1,C_1$ depending only on $(N,C,M,d)$ such that for any $x\in\T$, if $N(x,n)\leq M$, then we can find an $(N_1, C_1, 2^{-n}r)$-nice gluing circuit in $D^{(n)}(x)$ around $x$.
\end{lemma}
\begin{proof}
    We set $L := L_1$ and $R:= L_2$. Starting with a nice gluing circuit around $h^n(x)$ in $D(h^n(x))$, we will iterate the following algorithm to get a gluing circuit around $x$. 

    For $i\in [n]$, suppose we have an $(N,C,r)$-nice gluing circuit in $D^{(i)}(h^{n-i}(x))$ around $h^{n-i}(x)$. We have the following three scenarios:\begin{enumerate}
        \item~\label{first} When $\alpha$ is outside $D^{(i)}(h^{n-i}(x))$ or is between some $A_j$ and $A'_{j+1}$:
        
        In this case, suppose $h^{n-i-1}(x) = \widetilde L (h^{n-i}(x))$. Note that since both rescaled logarithmic energy and $\approx_\alpha$ are preserved under $\widetilde{L}$, $\widetilde L$ maps the $(N,C,r)$-nice gluing circuit around $h^{n-i}(x)$ to an $(N,C,r/2)$-nice gluing circuit around $h^{n-i-1}(x)$ in $D^{(i+1)}(h^{n-i-1}x) = \widetilde{L}(D^{(i)}(h^{n-i}x))$. The case for $\widetilde R$ is similar. Thus we get an $(N,C,r/2)$-nice gluing circuit around $h^{n-i-1}(x)$ in $D^{(i+1)}(h^{n-i-1}(x))$.

        \begin{figure}[ht]
            \centering
            \begin{tikzpicture}[scale=0.8]
                \draw [line width=1pt] (0,0) circle (3);
                \draw [dashed, line width=2pt, red] ({3*cos(222.5)},{3*sin(222.5)}) node [red, below left] {$\star_2$} -- ({3*cos(42.5)},{3*sin(42.5)}) node [red, above right] {$\star_1$};
                \fill [red] ({3*cos(85)},{3*sin(85)}) circle (2pt) node [above=0.6pt] {$\alpha$};
                \fill [teal] ({3*cos(168)},{3*sin(168)}) circle (2pt);
                \draw [->,teal, line width=1.2pt] ({2.6*cos(168)+0.3},{2.6*sin(168)}) node [right=0.6pt]{$h^{n-i}(x)$} --({2.8*cos(168)},{2.8*sin(168)});
                \fill [teal] ({3*cos(264)},{3*sin(264)}) circle (2pt) node [above =0.5pt] {$h^{n-i-1}(x)$};
            
                \draw [->,line width=3pt] (4,0) -- (5,0) node[pos=0.4,above=1pt] {\Large $\widetilde R$};
            
                \draw [line width=1pt] (9,0) circle (3);
                \draw [dashed, line width=2pt, red] ({9+3*cos(222.5)},{3*sin(222.5)}) node [red, below left] {$\star_2$} -- ({9+3*cos(42.5)},{3*sin(42.5)}) node [red, above right] {$\star_1$};
                \fill [red] ({9+3*cos(85)},{3*sin(85)}) circle (2pt) node [above=0.6pt] {$\alpha$};
                \fill [teal] ({9+3*cos(168)},{3*sin(168)}) circle (2pt);
                \fill [teal] ({9+3*cos(264)},{3*sin(264)}) circle (2pt);
            
                \gluepair{70}{81}{90}{99}{3}{0}{0}
                \gluepair{105}{115}{156}{165}{3}{0}{0}
                \gluepair{170}{184}{416}{426}{3}{0}{0}
                
                \gluepair{35}{40.5}{225}{229.5}{3}{9}{0}
                \gluepair{232.5}{237.5}{258}{262.5}{3}{9}{0}
                \gluepair{265}{272}{388}{393}{3}{9}{0}
            
                \nodeat{0}{0}{3.4}{75}{$A_3$}
                \nodeat{0}{0}{3.4}{95}{$A'_1$}
                \nodeat{0}{0}{3.4}{110}{$A_1$}
                \nodeat{0}{0}{3.4}{160}{$A'_2$}
                \nodeat{0}{0}{3.4}{178}{$A_2$}
                \nodeat{0}{0}{3.4}{60}{$A'_3$}
            \end{tikzpicture}
            \caption{Case~\ref{first} where $h^{n-i-1}(x)=\widetilde R (h^{n-i}(x))$}
            \label{fig:first}
        \end{figure}
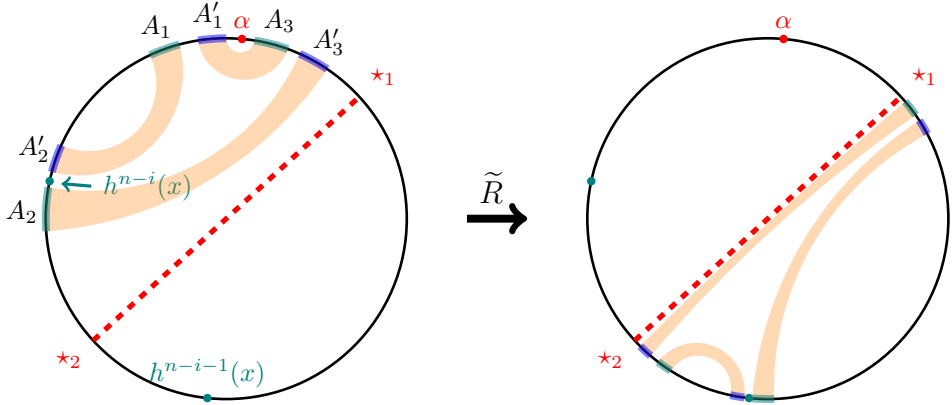

        \item~\label{second}  When $\alpha$ is in $D^{(i)}(x)$ and between some $A_j$ and $A'_j$:
        
        W.l.o.g., we assume $j=1$ and denote this arc between $A_1$ and $A'_1$ by $\Gamma_\alpha$. Then $h^{-1}(\Gamma_\alpha)$ consists of two arcs $J_1,J_2$ of length $|\Gamma_\alpha|/2$ contained in $D^{(i+1)}(h^{n-i-1}(x))$ covering $\star_1,\star_2$ respectively. Moreover, the gluing link $D^{(i+1)}(h^{n-i-1}(x))$ also contains all $\widetilde L(A_j)$, $\widetilde L(A'_j)$, $\widetilde R(A_j)$ and $\widetilde R (A'_j)$. Note that $\widetilde L(B_j)$ and $\widetilde L(B'_{j+1})$ are glued by $\widetilde L\circ \phi_i\circ h$ and same for $\widetilde R(B_j)$ and $\widetilde R(B'_{j+1})$. We have the corresponding gluing measures which are push-forwards of the original measures by $\widetilde L$ and $\widetilde R$. Going counterclockwise from $\widetilde L(A'_1)$, $\widetilde L(A_1)$ to $\widetilde R(A'_1)$, $\widetilde R(A_1)$, we denote the arcs as $C'_1,C_2, C'_2,\dots$, $C'_{2n}, C_1$ and the push-forward measures by $\nu'_1, \nu_2,\dots ,\nu'_{2n},\nu_1$ respectively. Since $\widetilde L$ and $\widetilde R$ are linear, the push forward measures have the same rescaled logarithmic energy as before and we have a measure preserving map gluing each $C_j$ with $C'_{j+1}$. Moreover, apart from $(C_1,C'_1)$ and $(C_n,C'_n)$ which are connected by $J_1$ and $J_2$ respectively, all other $C_j \sqcup C'_j$ are rescaled images of some $A_j\sqcup A'_j$ by $\widetilde L$ or $\widetilde R$. Thus we have $\diam(C_j\cup C'_j) \in [r/(2C),rC/2]$ for all $j$. This implies that $\set{C_j}\cup\set{C'_j}$ form a $(2N, C,r/2)$-nice gluing circuit around $h^{n-i-1}(x)$ in $D^{(i+1)}(x)$.

        \begin{figure}[ht]
            \centering
            \begin{tikzpicture}[scale=0.8]
                \draw [line width=1pt] (0,0) circle (3);
                \draw [dashed, line width=2pt, red] ({3*cos(222.5)},{3*sin(222.5)}) node [red, below left] {$\star_2$} -- ({3*cos(42.5)},{3*sin(42.5)}) node [red, above right] {$\star_1$};
                \fill [red] ({3*cos(85)},{3*sin(85)}) circle (2pt) node [above=0.6pt] {$\alpha$};
                \fill [teal] ({3*cos(168)},{3*sin(168)}) circle (2pt);
                \draw [->,teal, line width=1.2pt] ({2.6*cos(168)+0.3},{2.6*sin(168)}) node [right=0.6pt]{$h^{n-i}(x)$} --({2.8*cos(168)},{2.8*sin(168)});
                \fill [teal] ({3*cos(264)},{3*sin(264)}) circle (2pt) node [below =0.5pt] {$h^{n-i-1}(x)$};
                \fill [teal] ({9+3*cos(264)},{3*sin(264)}) circle (2pt) node [below =0.5pt] {$h^{n-i-1}(x)$};
            
                \draw [->,line width=3pt] (4,0) -- (5,0) node[pos=0.4,above=2pt] {\Large $\widetilde L~\&~\widetilde R$};
            
                \draw [line width=1pt] (9,0) circle (3);
                \draw [dashed, line width=2pt, red] ({9+3*cos(222.5)},{3*sin(222.5)}) node [red, below left] {$\star_2$} -- ({9+3*cos(42.5)},{3*sin(42.5)}) node [red, above right] {$\star_1$};
                \fill [red] ({9+3*cos(85)},{3*sin(85)}) circle (2pt) node [above=0.6pt] {$\alpha$};
                \fill [teal] ({9+3*cos(168)},{3*sin(168)}) circle (2pt);
                \fill [teal] ({9+3*cos(264)},{3*sin(264)}) circle (2pt);
                \draw [red, line width=5pt, opacity=0.5] ({3*cos(82)},{3*sin(82)}) arc(82:90:3);
                \draw [red, line width=5pt, opacity=0.5] ({9+3*cos(41)},{3*sin(41)}) arc(41:45:3);
                \draw [red, line width=5pt, opacity=0.5] ({9+3*cos(221)},{3*sin(221)}) arc(221:225:3);
            
                \gluepair{90}{100}{111}{118}{3}{0}{0}
                \gluepair{125}{134}{156}{165}{3}{0}{0}
                \gluepair{170}{184}{430}{442}{3}{0}{0}
                
                \gluepair{45}{50}{55.5}{59}{3}{9}{0}
                \gluepair{225}{230}{235.5}{239}{3}{9}{0}
                \gluepair{62.5}{67}{78}{82.5}{3}{9}{0}
                \gluepair{242.5}{247}{258}{262.5}{3}{9}{0}
                \gluepair{85}{92}{215}{221}{3}{9}{0}
                \gluepair{265}{272}{395}{401}{3}{9}{0}
            
                \nodeat{0}{0}{3.4}{95}{$A_1$}
                \nodeat{0}{0}{3.4}{115}{$A'_2$}
                \nodeat{0}{0}{3.4}{130}{$A_2$}
                \nodeat{0}{0}{3.4}{160}{$A'_3$}
                \nodeat{0}{0}{3.4}{178}{$A_3$}
                \nodeat{0}{0}{3.4}{74}{$A'_1$}
            \end{tikzpicture}
            \caption{Case~\ref{second} where $h^{n-i-1}(x)=\widetilde R (h^{n-i}(x))$, arcs $\Gamma_\alpha$, $J_1$ and $J_2$ are highlighted in red.}
            \label{fig:second}
        \end{figure}
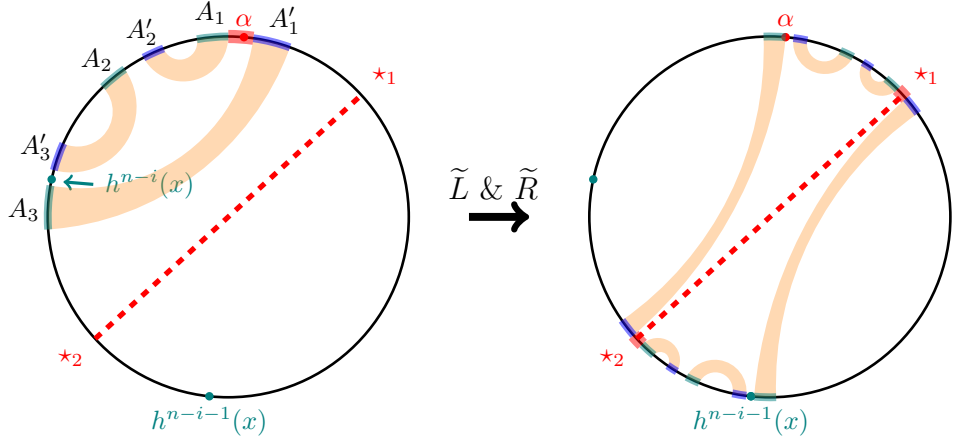
        
        \item \label{third} When $\alpha$ is in some $A_j$ or $A'_j$:
        
        We might assume the former case as the argument works for the latter with little modification. Thus we can also assume that $\alpha$ is in $A_1$.
        We cut $A_1$ at $\alpha$ into two sets $C_1,D_1$ such that $\alpha$ is between $C_1$ and $A'_2$ and also between $D_1, A'_1$. Since $\phi_1$ respects the lamination, it has to be monotone. We extend it arbitrarily to be a monotone function from $A_1$ to $A'_2$ and take $\alpha' = \phi_1(\alpha)$. We can then cut $A'_2$ at $\alpha'$ into two sets $C'_2$ and $D'_2$ such that $\phi_1(C_1)\subseteq C'_2$ and $\phi_1(D_1)\subseteq D'_2$ (see Figure~\ref{fig:cutting}).

        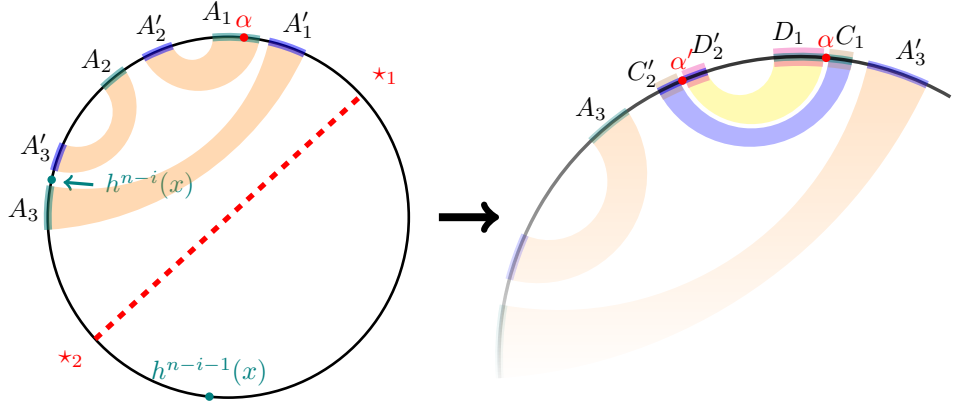
\begin{figure}[ht]
            \centering
            \begin{tikzpicture}[scale=0.8]
                \draw [line width=1pt] (0,0) circle (3);
                \draw [dashed, line width=2pt, red] ({3*cos(222.5)},{3*sin(222.5)}) node [red, below left] {$\star_2$} -- ({3*cos(42.5)},{3*sin(42.5)}) node [red, above right] {$\star_1$};
                \fill [teal] ({3*cos(168)},{3*sin(168)}) circle (2pt);
                \draw [->,teal, line width=1.2pt] ({2.6*cos(168)+0.3},{2.6*sin(168)}) node [right=0.6pt]{$h^{n-i}(x)$} --({2.8*cos(168)},{2.8*sin(168)});
                \fill [teal] ({3*cos(264)},{3*sin(264)}) circle (2pt) node [above =0.5pt] {$h^{n-i-1}(x)$};
                
                \draw [->,line width=3pt] (3.5,0) -- (4.5,0);
            
                \gluepair{80}{95}{108}{118}{3}{0}{0}
                \gluepair{125}{134}{156}{165}{3}{0}{0}
                \gluepair{170}{184}{425}{437}{3}{0}{0}

                \gluepairc{80}{84.5}{113.5}{118}{5}{9.5}{2-2.5*sqrt(3)}{brown}{blue}
                \gluepairc{85.5}{95}{108}{112.5}{5}{9.5}{2-2.5*sqrt(3)}{magenta}{yellow}
            
                \draw [teal, line width=3pt, opacity=0.5] ({9.5+5*cos(80)},{2-2.5*sqrt(3)+5*sin(80)}) arc(80:95:5);
            
                \draw [blue, line width=3pt, opacity=0.5] ({9.5+5*cos(108)},{2-2.5*sqrt(3)+5*sin(108)}) arc(108:118:5);
            
                \begin{scope}
                    \path [scope fading=south] (6,-3) rectangle (12,4);
                    \draw [line width=1.3pt] (12,2) arc(60:184:5);
                    \gluepair{125}{134}{156}{165}{5}{9.5}{2-2.5*sqrt(3)}
                    \gluepair{170}{184}{425}{437}{5}{9.5}{2-2.5*sqrt(3)}
                \end{scope}
            
                \nodeat{0}{0}{3.4}{93}{$A_1$}
                \nodeat{0}{0}{3.4}{112}{$A'_2$}
                \nodeat{0}{0}{3.4}{130}{$A_2$}
                \nodeat{0}{0}{3.4}{160}{$A'_3$}
                \nodeat{0}{0}{3.4}{178}{$A_3$}
                \nodeat{0}{0}{3.4}{70}{$A'_1$}
            
                \nodeat{9.5}{2-2.5*sqrt(3)}{5.4}{81}{$C_1$}
                \nodeat{9.5}{2-2.5*sqrt(3)}{5.4}{92}{$D_1$}
                \nodeat{9.5}{2-2.5*sqrt(3)}{5.4}{106.5}{$D'_2$}
                \nodeat{9.5}{2-2.5*sqrt(3)}{5.4}{119}{$C'_2$}
                \nodeat{9.5}{2-2.5*sqrt(3)}{5.4}{130}{$A_3$}
                \nodeat{9.5}{2-2.5*sqrt(3)}{5.4}{70}{$A'_3$}

                \fill [red] ({9.5+5*cos(85)},{2-2.5*sqrt(3)+5*sin(85)}) circle (2pt) node [above=0.6pt] {$\alpha$};
                \fill [red] ({9.5+5*cos(113)},{2-2.5*sqrt(3)+5*sin(113)}) circle (2pt) node [above=1pt] {$\alpha'$};
                \fill [red] ({3*cos(85)},{3*sin(85)}) circle (2pt) node [above=0.6pt] {$\alpha$};
                
            \end{tikzpicture}
            \caption{For Case~\ref{third}, we cut $A_1$ into two sets along $\alpha$ and cut $A'_2$ in the compatible way.}
            \label{fig:cutting}
        \end{figure}

        We let $G_1$ and $G_2$ be the two components of $h^{-1}(A_1)$ containing $\star_1$ and $\star_2$ respectively. Then going counterclockwise, the preimage of $(\cup A_i) \cup (\cup A'_i)$ under $h$ consists of $2n$ disjoint closed arcs (see Figure~\ref{fig:flowing}) \begin{equation*}
            G_1, \widetilde L(A'_2), \widetilde L(A_2),\dots, \widetilde L(A'_1), G_2, \widetilde R(A'_2), \widetilde R(A_2),\dots , \widetilde R(A'_1).
        \end{equation*} We also have that \begin{align*}
            \diam G_1 = \diam &G_2 = \frac{1}{2}\diam A_1,\\
            \diam(G_1\cup \widetilde R(A'_1)) =  \diam(G_2\cup &\widetilde L(A'_1)) = \frac{1}{2}\diam (A_1\cup A'_1).
        \end{align*}

               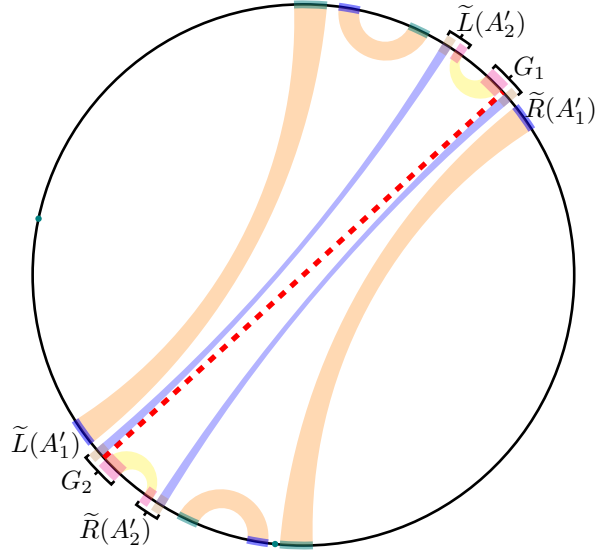
\begin{figure}[ht]
            \centering
            \begin{tikzpicture}[scale=0.6]
                \draw [line width=1pt] (0,0) circle (6);
                \draw [dashed, line width=2pt, red] ({6*cos(222.5)},{6*sin(222.5)})  -- ({6*cos(42.5)},{6*sin(42.5)}) ;

                \fill [teal] ({6*cos(168)},{6*sin(168)}) circle (2pt);
                \fill [teal] ({6*cos(264)},{6*sin(264)}) circle (2pt);
                
                \gluepairc{57}{59}{220}{222}{6}{0}{0}{brown}{blue}
                \gluepairc{237}{239}{400}{402}{6}{0}{0}{brown}{blue}
                \gluepairc{43}{47.5}{54}{56}{6}{0}{0}{magenta}{yellow}
                \gluepairc{223}{227.5}{234}{236}{6}{0}{0}{magenta}{yellow}
                \gluepair{62.5}{67}{78}{82.5}{6}{0}{0}
                \gluepair{242.5}{247}{258}{262.5}{6}{0}{0}
                \gluepair{85}{92}{212.5}{218.5}{6}{0}{0}
                \gluepair{265}{272}{392.5}{398.5}{6}{0}{0}

                \curlybrace{40}{47.5}{6.3}
                \curlybrace{220}{227.5}{6.3}
                \curlybrace{54}{59}{6.3}
                \curlybrace{234}{239}{6.3}
                
                \nodeat{0}{0}{6.8}{42.5}{$G_1$}
                \nodeat{0}{0}{6.8}{222.5}{$G_2$}

                \nodeat{0}{0}{6.8}{213}{$\widetilde L (A'_1)$}
                \nodeat{0}{0}{6.8}{33}{$\widetilde R(A'_1)$}
                
                \nodeat{0}{0}{7}{53.5}{$\widetilde L (A'_2)$}
                \nodeat{0}{0}{7}{233.5}{$\widetilde R(A'_2)$}
                
            \end{tikzpicture}
            \caption{The preimages of $A_i, A'_i, C_1, D_1, C'_2, D'_2$. If $\mu_1(C_1)$ is large (blue tunnels are more "conductive"), we choose the gluing circuit in one semi-circle similar to Case~\ref{first}. If $\mu_1(D_1)$ is large (yellow tunnels are more "conductive"), we glue two gluing circuits together like Case~\ref{second}.}
            \label{fig:flowing}
        \end{figure}
        W.l.o.g we assume that $h^{n-i-1}(x) \in C(L)$. Then we have the following two scenarios:\begin{enumerate}
            \item $\mu_1(C_1)\geq 1/2$ and thus $\mu'_2(C'_2)\geq 1/2$:\\
            We set $m_1 = \mu_1|_{C_1}/\mu_1(C_1)$ and $m'_2 = \mu_2|_{C'_2}/\mu'_2(C'_2)$. Then $m_1$ as a measure on $A_1$ (and $m_2'$ on $A'_2$) has rescaled logarithmic energy bounded by $4C$. We let $\nu_1$ and $\nu_2$ be the push-forwards of $m_1$ and $m'_2$ under $\widetilde L$. Additionally, $\widetilde L\circ \phi_1 \circ h: \widetilde L(C_1)\to \widetilde L(C'_2)$ will be $(\nu_1,\nu_2)$ measure-preserving. It follows that \begin{equation*}
                (\widetilde L(A'_1), G_2), (\widetilde L(A'_2), \widetilde L(A_2)), \dots, (\widetilde L(A'_n), \widetilde L(A_n))
            \end{equation*}
            form a $(N,4C,r/2)$-nice gluing circuit around $h^{n-i-1}(x)$.
            
            \item $\mu_1(D_1) > 1/2$ and thus $\mu'_2(D'_2) > 1/2$:\\
            We set $m_1 = \mu_1|_{D_1}/\mu_1(D_1)$ and $m'_2 = \mu_2|_{D'_2}/\mu'_2(D'_2)$. Then $m_1$ as a measure on $A_1$ (and $m_2'$ on $A'_2$) has rescaled logarithmic energy bounded by $4C$. We let $\nu^L_1$ and $\nu^L_2$ be the pushforward of $m_1$ and $m'_2$ under $\widetilde L$. Then $\nu^L_1$ is supported on $G_1$ and $\nu^L_2$ on $\widetilde L(A'_2)$. We similarly define $\nu^R_1$ and $\nu^R_2$ supported on $G_2$ and $\widetilde R(A'_2)$ respectively. Note that all of the four have rescaled logarithmic energy bounded by $4C$ and $\widetilde L\circ \phi_1 \circ h: \widetilde L(D_1)\to \widetilde L(D'_2)$ is $(\nu^L_1,\nu^L_2)$ measure-preserving and $\widetilde R\circ \phi_1 \circ h:\widetilde R(D_1)\to \widetilde R(D'_2)$ is $(\nu^R_1,\nu^R_2)$ measure-preserving. It follows that \begin{align*}
                (\widetilde R(A'_1), G_2), (\widetilde L(A'_2), \widetilde L(A_2)), \dots, (\widetilde L(A'_n), \widetilde L(A_n)),\\
                (\widetilde L(A'_1), G_2), (\widetilde R(A'_2), \widetilde R(A_2)), \dots, (\widetilde R(A'_n), \widetilde R(A_n)),\\
            \end{align*}
            form a $(2N,4C,r/2)$-nice gluing circuit around $h^{n-i-1}(x)$.
        \end{enumerate}
    \end{enumerate}

    In summary, each iteration, we will divide $r$ by $2$. If $\alpha$ is outside $D^{(i)}(x)$, we keep $N,C$ unchanged. Otherwise, at worst we need to increase $N$ to $2N$ and increase $C$ to $4C$. Since $N(x,n)<M$, $\alpha$ is in at most $M$ of $D^{(i)}(h^{n-i}(x))$'s for $k=1,\dots, n$. Thus we will end up with a $(2^M N, 4^M C, 2^{-n}r)$-nice gluing circuit around $x$.
\end{proof}

\begin{proof}[Proof of Theorem~\ref{main_2}]
    Let $\set{D(x)}$ be a covering of $\T$ by gluing links satisfies the two conditions as in Definition~\ref{CCE}. That is, there exist $(N,C,r)$ such that each $D(x)$ contains an $(N,C,r)$-nice gluing circuit around $x$. Moreover, there exist $(P,M)$ such that for each $x\in \T$, we have a sequence $\set{n_j}$ with $n_j\leq Pj$ such that $N(x,n_j) \leq M$. By Lemma~\ref{copying}, there exist $(N_1,C_1)$ such that there exists $(N_1,C_1, 2^{-n_j})$-nice gluing circuit around $x$ contained in $D^{(n_j)}(x)$ for each $n_j$. The result then follows from Theorem~\ref{peter_gluing}.
\end{proof}

\section{Strong Recurrence and Digit-Fixing Set} \label{SR_sec}
In this section, we will discuss how to guarantee low encounter number $N(x,n)$ for a covering $\set{D(x)}_{x\in \T}$. We will define the \textit{digit-fixing} condition inspired. To define this condition, we will first introduce the \textit{strong recurrence} condition.

\subsection{Strong Recurrence}
    \begin{definition}[Duplicating intervals and digits]
    Let $\alpha$ be fixed. Given any finite or infinite word $g$, we define the following.
    \begin{enumerate}
        \item  Fix a set $\mathcal{R}\subseteq \N$. We say an interval $[a,b]$ of digits of $g$ is $\mathcal{R}$-\textit{duplicating} if $g|_{[a,b]}$ coincides with $I^\alpha(\alpha)|_{[1,b-a+1]}$ except possibly in the digits of the former contained in $\mathcal{R}\cap [a,b]$. We also call a $\emptyset$-duplicating interval a \textit{(proper) duplicating interval}.
        \item Given an integer $D$ and a set $\mathcal{R}\subseteq \N$. We say the $i$-th digit of $g$ is $(D,\mathcal{R})$\textit{-duplicating} if $i$ is covered by some $\mathcal{R}$-duplicating interval $I$ such that $I$ either covers the last letter of $g$ or has length exceeding $D$.
    \end{enumerate}
\end{definition}

 Here we introduce the notion of strong recurrence by S. Smirnov in  \cite{smirnov2000symbolic}. We say a set $\mathcal{R} \subseteq \N$ is $D$-\textit{rare} if $\#(\mathcal{R} \cap [i,i+D]) \leq 3$ for any $i\in \N$. 

\begin{definition}[Strong recurrence \cite{smirnov2000symbolic}] 
    Given an angle $\alpha\in \T$ and its itinerary $I^\alpha(\alpha)$, we say $\alpha$ is \textit{strongly recurrent (SR)} if for every $D > 0$ and $\tau<1$, there exist integer $n>D$ and a $D$-rare set $\mathcal{R}_n \subseteq [n]$ such that \begin{equation*}
    \#\set{0 < i \leq  n: \text{$i$-th digit of $I^\alpha(\alpha)|_{[1,n]}$ is~$(D,\mathcal{R}_n)$-duplicating}} > \tau n.
    \end{equation*} 
\end{definition}

Intuitively, if $\alpha$ is strongly recurrent, then the word $I^\alpha(\alpha)$ frequently repeats its starting pattern. As a result, $I^\alpha(\alpha)$ should "contain less information" than a generic infinite word in $\set{L_1,\dots,L_d}^\infty$. We have the following result regarding the rarity of strongly recurrent angles:

\begin{proposition}[\cite{smirnov2000symbolic}]\label{SR}
    The set of SR angles has Hausdorff dimension zero.
\end{proposition}

Note that in Smirnov's original paper, a $D$-rare set is defined as a set $\mathcal{R}\subseteq \N$ with the property that $\#(\mathcal{R}\cap [i,i+D])\leq 2$ for any $i$. For our modified definition, one can check that the proof for this Proposition still works by essentially changing some constants.

\subsection{Digit-Fixing Property}\label{df}
Recall that given a covering $\set{D(x)}_{x\in\T}$ of $\T$ by gluing links and a integer $n>1$, we can construct a new covering by (usually smaller) gluing links by setting \begin{equation*}
    D^{(n)}(x) = \comp_x^\alpha h^{-n}(D(h^n(x))).
\end{equation*}

\begin{definition}[digit-fixing cover]
    We say the covering $\set{D(x)}_{x\in\T}$ of gluing links is $L$-\textit{Digit-Fixing}, if the following condition holds:\begin{quote}
        For each $n$, if both $x_1,x_2$ are points in $D^{(n)}(y)$ for some $ y \in\T$, then the words $I^\alpha(x_1)|_{[1,n+1]}$ and $I^\alpha(x_2)|_{[1,n+1]}$ differ only on an increasing sequence of digits $\set{n_i}$ (depending on the choice of $x_1$ and $x_2$) such that $n_{i+1}-n_i > L$.
    \end{quote}
\end{definition} 

Intuitively speaking, if the covering $\set{D(x)}_{x\in\T}$ is digit-fixing, then each $D^{(n)}(x)$ only contains points with almost the same itinerary up to $(n+1)$-th digit. This definition is inspired by the following observation.

\begin{lemma}[{\cite[Lemma~2.1]{smirnov2000symbolic}}]\label{smir}
    Let $J\subseteq \C$ be a dendrite Julia set for polynomial $p_c(z):= z^d+c$. Then for any integer $L$, there exists $r>0$ such that for any $r$-radius ball $B_r$ that intersects $J$, if points $x, y \in \T$ are such that $\varphi_c(x)$ and $\varphi_c(y)$ belong to the same component of $p_c^{-n}(B_r)$, then $I^\alpha(x)|_{[1,n+1]}$ and $I^\alpha(y)|_{[1,n+1]}$ differ only on an increasing sequence of digits $\set{n_i}$ such that $n_{i+1}-n_i > L$.
\end{lemma}

In a way, the lemma implies that small Euclidean balls intersecting the Julia set are digit-fixing. Again, in our setting, the gluing links play the role of Euclidean balls and the map $h$ plays the role of the polynomial $p_c$.

\subsection{Criteria for Strong Recurrence and Digit-Fixing}
Essentially only using Lemma~\ref{smir}, S. Smirnov proves that the special angle of a non-TCE polynomial satisfies SR condition. Therefore, we can extract from the result a more general criterion for SR. To do so, we first introduce the definition of semi-CCE condition:

\begin{definition}\label{semi-CCE}
    Given $\alpha\in\T$ and $h:z\to z^d$, we says $\alpha$ satisfies \textit{semi-CCE condition} if there exist $(L,P,M)$ such that the following is satisfied:
    \begin{quote}
        If $\{D(y)\}_{y\in \T}$ is a $L$-digit-fixing covering of gluing links, then for any $x\in\T$ there exists an increasing sequence of integers $\set{n_j}$ with $n_j\leq Pj$ and $N(x,n_j) \leq M$.
    \end{quote}
\end{definition}

In other words, semi-CCE condition guarantees that any digit-fixing cover of gluing links will give low encounter numbers (recall we need this for Condition~2 of Definition~\ref{CCE}). The relation between semi-CCE and SR is as follows.

\begin{theorem}[\cite{smirnov2000symbolic}]\label{smirn}
    If $\alpha$ is not pre-periodic and does not satisfy semi-CCE condition, it is strongly recurrent. 
\end{theorem}

Here we repeat the proof in Smirnov's paper for the general setting.

\begin{proof}
    We first observe that since $\alpha$ is aperiodic, the collection of equivalence classes $\set{[x]_{\approx_\alpha}}_{x\in\T}$ are gluing links covering $\T$ that are digit-fixing for any $L$. Hence, there always exist some $L$-digit-fixing cover $\set{D(x)}_{x\in\T}$. Given arbitrary $D>0$ and $\tau <1$, we shall show that if $\alpha$ does not satisfy semi-CCE, then there exist integer $n>D$ and a $D$-rare set $\mathcal{R} \subseteq [n]$ such that \begin{equation*}
        \# \set{0 < i\leq  n: \text{$i$-th digit of $I^\alpha(\alpha)|_{[1,n]}$ is~$(D,\mathcal{R})$-duplicating}} >\tau n.
    \end{equation*}
    We let $L=D+1$ and set $M,P$ such that $M>2D+2$ and $(1-\tau)(M/2)(1-1/P)>1$. Since $\alpha$ does not satisfy semi-CCE, there exists $L$-digit-fixing cover $\set{D(y)}_{y\in T}$ and a point $x\in \T$ and $N>0$ such that at least $(1-1/P)N$ integers $n\in [N]$ satisfy $N(x,n)>M$. Now we define the following function :\begin{align*}
        l:&[N-1]\cup \set{0} \to [N] \cup \set{-\infty}\\
        &i\mapsto\max \set{i\leq j \leq N: \alpha \in D^{(j-i)}(h^i(x))},
    \end{align*} where we set $\max \emptyset = -\infty$. In other words, $l(i)$ is the largest integer $j$ not exceeding $N$ such that the gluing link of $h^i(x)$ we get from pulling back $D(h^j(x))$ contains $\alpha$. Since $\set{D(x)}$ is $L$-digiting fixing, it follows that $I^\alpha(h^i(x))_{[1,l(i)-i+1]} (= I^\alpha(x)|_{[i,l(i)]})$ and $I^\alpha(\alpha)|_{[1,l(i)-i+1]}$ match except possibly at a set of digits $S_i\subseteq [i,l(i)]$, where digits in $S_i$ are at least $D$ digits apart. We set $\mathcal{J}= \set{[i,l(i)]}_{i=0}^{N-1}$. 
    
    Now we observe that if $N(x,n)> M$ for some  integer $n$, then $\alpha$ is contained in at least $(M+1)$ of $D^{(i)}(h^{n-i}(x))'s$. When $\alpha \in D^{(i)}(h^{n-i}(x))$, we know that $l(n-i) \geq n$ and thus $n\in [n-i,l(n-i)]$. Hence, it follows that $\mathcal{J}$ covers the integer $n$ at least $M+1$ times. Therefore, since at least $(1-1/P)N$ integers in $[N]$ satisfy $N(x,n) > M$, $\mathcal{J}$ covers at least $(1-1/P)N$ integers at least $(M+1)$ times.
    
    Note that all the intervals in $\mathcal{J}$ start at different points. Thus we can order the intervals by their left ends. By removing intervals in $\mathcal{J}$ with length not greater than $D$, then we can a new collection $\mathcal{J}'$ that covers those integers at least $(1-1/P)N$ integers in $[N]$ at least $[M+1-(D+1)]> M/2$ times. Thus we have \begin{equation}\label{contract}
        \sum_{I\in\mathcal{J}'} |I|> \frac{M}{2}\cdot  \left(1-\frac{1}{P}\right)N \geq \frac{N}{1-\tau}.
    \end{equation}

    Given an interval $[i,l(i)]\in \mathcal{J}'$, we set \begin{equation*}
        \mathcal{J}_i:= \set{[j,l(j)] \in \mathcal{J}': i< j \leq l(i)}
    \end{equation*} to be the collection of intervals in $\mathcal{J'}$ intersecting $[i,l(i)]$ from the right. The standard Besicovitch covering argument implies that there exists $\mathcal{J}'_i \subseteq \mathcal{J}_i$ such that $\mathcal{J}'_i$ covers the same integers in $\mathcal{J}_i$ but each integer is covered at most twice. Note that by the definition of $S_k$ above, we have \begin{equation*}
        I^\alpha(x)[m] = I^\alpha(\alpha)[m-k+1]\quad \text{for all} \quad m\in [k,l(k)]\backslash S_k.
    \end{equation*}
    It follows that for any $[j,l(j)]$ in $\mathcal{J}'_i$, $I^\alpha(\alpha)|_{[i,l(i)]\cap [j,l(j)]-(i-1)}$ matches well with $I^\alpha(x)|_{[i,l(i)]\cap [j,l(j)]}$ (except at $S_i$), which in turn matches well with $I^\alpha(\alpha)|_{[j,l(j)]-(j-1)}$ (except at $S_j$). Therefore, the interval $[i,l(i)]\cap [j,l(j)]-(i-1)$ is $(S_i\cap S_j -(i-1))$-duplicating for $I^\alpha(\alpha)$. Since $\mathcal{J}'_i$ covers each integer at most twice, the set \begin{equation*}
         \mathcal{R} := S_i \cup \left(\bigcup_{[j,l(j)]\in \mathcal{J}'_i}S_j\right)
    \end{equation*} is rare and so is $\mathcal{R}-(i-1)$. Therefore, for each $[j,l(j)]\in \mathcal{J}'_i$, $[i,l(i)]\cap [j,l(j)]-(i-1)$ is a $(\mathcal{R}-(i-1))$-duplicating interval of $I^\alpha(\alpha)|_{[i,l(i)]-(i-1)}$, either of length greater than $D$ or containing the $(l(i)-i+1)$-th digit. If $\mathcal{J}_i$ (thus $\mathcal{J}'_i$) covers more than $\tau(l(i)-i)$ integers in $[i,l(i)]$, then more than $\tau(l(i)-i)$ digits of $I^\alpha(\alpha)|_{[1,l(i)-i+1]}$ are $(D,\mathcal{R}-(i-1))$-duplicating and we are done (recall that $l(i)-i>D$).

    To recapitulate, if there exists $[i,l(i)]\in \mathcal{J}'$ such that $\mathcal{J}_i$ defined above covers more than $\tau(l(i)-i)$ integers, then we are done. If such an interval does not exist, then for all $[i,l(i)]$'s, at most $\tau(l(i)-i)$ integers in the interval can be covered by some $[j,l(j)]$ with $j>i$. Given an integer $k$ in $[N]$, we denote the number of intervals in $\mathcal{J}'$ covering $k$ by $m(k)$. Then among all the intervals covering $k$, only for the rightmost one, $k$ is not covered by any interval right to it. This implies that \begin{align*}
         \sum_{k=1}^N(m(k)-1)
        =&\sum_{I\in\mathcal{J}'}\# \{i\in I: \text{$i$ is covered by some interval right to $I$}\}\\
        \leq &\sum_{I\in\mathcal{J}'} \tau |I|
        =\tau \sum_{I\in\mathcal{J}'} |I|.
    \end{align*} Using the fact that $\sum_{k=1}^Nm(k) = \sum_{I\in\mathcal{J}'} |I|$, we conclude $\
        \sum_{I\in\mathcal{J}'} |I| \leq N/(1-\tau)$, which contradicts $\eqref{contract}$. Thus the result follows.
\end{proof}

\begin{corollary}
    Semi-CCE condition is satisfied by all points on $\T$ except for a set of Hausdorff dimension zero.
\end{corollary}

Here we also provide a criterion to check whether a covering ${D(x)}_{x\in\T}$ is digit-fixing. This criterion is relatively easy to check for the covering we will use in Section~\ref{theo_sec}. It is inspired by the proof of Lemma~2.1 in \cite{smirnov2000symbolic}.

\begin{proposition}\label{digit_fixing}
    Let $\set{D(x)}_{x\in\T}$ be a cover of $\T$ by gluing links. If for all $x\in\T$ and integer $n$ with $\alpha \in D^{(n)}(x)$, we have
    \begin{equation*} 
        D^{(n)}(x) \cap \set{h^i(\alpha)}_{i=1}^L = \emptyset,
    \end{equation*} then $\set{D(x)}$ is $L$-digit-fixing.
\end{proposition}
\begin{proof}
    Given such a covering $\set{D(x)}$, suppose $x_1$ and $x_2$ are both in $D^{(n)}(y)$. If the $j$-th digits of $I^\alpha(x_1)$ and $I^\alpha(x_2)$ differ with $j\leq n+1$, then either one of $h^{j-1}(x_1)$ and $h^{j-1}(x_2)$ is in $\set{\star_1,\star_2}$ or the two are separated by $\set{\star_1,\star_2}$. Therefore, $h^{j-1}(D^{(n)}(y)) =D^{(n-j+1)}(h^{j-1}(y))$, which contains $\set{h^{j-1}(x_1),h^{j-1}(x_2)}$, has to intersect $\set{\star_1,\star_2}$ as well. It follows that $\alpha \in D^{(n-j)}(h^j(y))$. By our assumption, this would mean \begin{equation*}
        D^{(n-j)}(h^j(y)) \cap \set{h^i(\alpha)}_{i=1}^L = \emptyset.
    \end{equation*}
    Now if $I^\alpha(x_1)$ and $I^\alpha(x_2)$ also differ at the digit $k \in [j-L,j)$, by the same argument, we have $\alpha \in D^{(n-k)}(h^k(y))$ and thus \begin{equation*}
        h^{j-k}(\alpha) \in h^{j-k}(D^{(n-k)}(h^k(y))) = D^{(n-j)}(h^j(y)),
        \end{equation*}
        which contradicts our assumption. Thus the digits at which the two words differ have to be at least $L$ digit apart.
\end{proof}

\section{Generalized Cylinder Sets and Legal Words} \label{GCS_sec}
To recapitulate, we have shown that CCE condition implies the existence of Hölder continuous welding and non-SR implies semi-CCE. If we can find some gluing link covering $\set{D(x)}$ that satisfies the nice gluing property and is digit-fixing, then we can link two results. In the case when $d=2$, the candidate we have is the \textit{generalized cylinder sets}. In this section, we will present their definition and give a nice cover of $\D$ by such sets.

\subsection{Generalized Cylinder Sets}
From now on, we assume $d=2$ and again set $L=L_1$ and $R=L_2$. The definitions in this section can be generalized to $d\geq 3$ but for now we will restrict our attention for simplicity. 

\begin{lemma}
    Given a finite word $u\in\set{L,R}^*$, since $\alpha$ is aperiodic, $\widetilde u(\star_1)$ and $\widetilde u(\star_2)$ are well-defined. We set $l_u$ to be the hyperbolic geodesic in $\overline{\D}$ connecting $\widetilde u(\set{\star_1,\star_2})$. Then for $u_1\neq u_2$, the two geodesics $l_{u_1}$ and $l_{u_2}$ are disjoint.
\end{lemma}

\begin{proof}
    Suppose we have $u_1\neq u_2$ such that $l_{u_1}$ and $l_{u_2}$ intersect. By the non-crossing property in Lemma~\ref{non-crossing}, we have \begin{equation*}
        \widetilde{u_1}(\star_1) \approx_\alpha \widetilde{u_2}(\star_1).
    \end{equation*} However, we have \begin{equation*}
        I^\alpha(\widetilde{u_1}(\star_1)) = u_1\star I^\alpha(\alpha),\quad I^\alpha(\widetilde{u_2}(\star_1)) = u_2\star I^\alpha(\alpha).
    \end{equation*} This would imply that $u_1=u_2$, causing contradiction.
\end{proof}

\begin{definition}[Generalized Cylinder Sets]
    Given $l_u$ defined above, we set $L_n = \set{l_u: |u| = n}$. The geodesics in $L_n$ cut $\overline \D$ into several components. We denote this collection of components by $\mathcal{D}_n$ and set $\mathcal{GCS}_n = \set{D\cap \T: D\in\mathcal{D}_n}$. Since $\widetilde u(\star_1) \approx_\alpha \widetilde u(\star_2)$, one can easily check that $\mathcal{GCS}_n$ is a collection of gluing links. We call each member in this collection a \textit{generalized cylinder set of degree} $n$.
\end{definition}

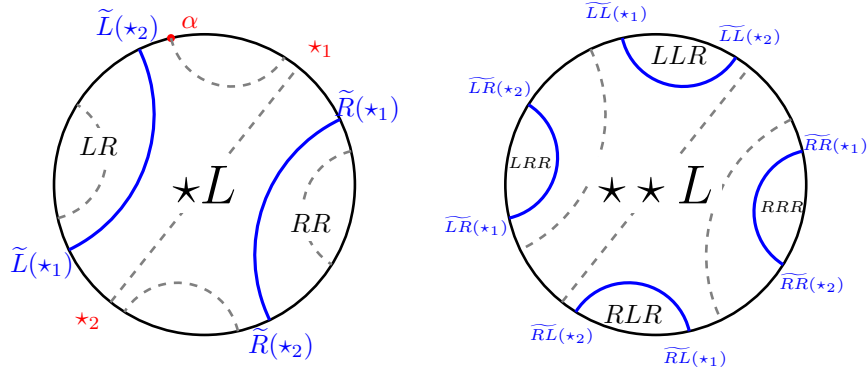
\begin{figure}[ht]
    \centering
    \begin{tikzpicture}[scale=0.5]
        \draw [line width=1pt] (0,0) circle (4);
        \fill [red] ({4*cos(720/7)},{4*sin(720/7)}) circle (3pt) node [red,above right=0.6pt]{$\alpha$};
    
        \draw [dashed, line width=1pt, gray] ({-4*cos(360/7)},{-4*sin(360/7)}) node [red, below left] {$\star_2$} -- ({4*cos(360/7)},{4*sin(360/7)}) node [red, above right] {$\star_1$};
        
        \harcc{0}{0}{4}{180/7}{360*9/28+180}{line width=1.2pt, blue}
        \harcc{0}{0}{4}{360*9/28}{180/7+180}{line width=1.2pt, blue}

        \harcc{0}{0}{4}{720/7}{180*9/28+360}{dashed, line width=1pt, gray}
        \harcc{0}{0}{4}{720/7-180}{180*9/28+180}{dashed, line width=1pt, gray}
        \harcc{0}{0}{4}{90/7}{180*23/28+180}{dashed, line width=1pt, gray}
        \harcc{0}{0}{4}{90/7+180}{180*23/28+360}{dashed, line width=1pt, gray}

        \nodeatc{0}{0}{0}{95}{\Huge $\star L$}{fill=white}
        \nodeatc{0}{0}{3}{160}{$LR$}{fill=white}
        \nodeatc{0}{0}{3}{340}{$RR$}{fill=white}
        
        \draw [line width=1pt] (12,0) circle (4);
    
        \draw [dashed, line width=1pt, gray] ({12-4*cos(360/7)},{-4*sin(360/7)}) -- ({12+4*cos(360/7)},{4*sin(360/7)});
        
        \harcc{12}{0}{4}{180/7}{360*9/28+180}{dashed, line width=1pt, gray}
        \harcc{12}{0}{4}{360*9/28}{180/7+180}{dashed, line width=1pt, gray}

        \harcc{12}{0}{4}{720/7}{180*9/28+360}{line width=1.2pt, blue}
        \harcc{12}{0}{4}{720/7-180}{180*9/28+180}{line width=1.2pt, blue}
        \harcc{12}{0}{4}{90/7}{180*23/28+180}{line width=1.2pt, blue}
        \harcc{12}{0}{4}{90/7+180}{180*23/28+360}{line width=1.2pt, blue}

        \nodeatc{12}{0}{0}{95}{\Huge $\star\star L$}{fill=white}
        \nodeat{12}{0}{3.5}{80}{$LLR$}
        \nodeat{12}{0}{3.5}{260}{$RLR$}
        \nodeat{12}{0}{3.4}{170}{\tiny$LRR$}
        \nodeat{12}{0}{3.4}{350}{\tiny $RRR$}
        
        \nodeatc{0}{0}{4.8}{180/7}{$\widetilde{R}(\star_1)$}{color = blue}
        \nodeatc{0}{0}{4.7}{360*9/28+180}{$\widetilde{R}(\star_2)$}{color = blue}
        \nodeatc{0}{0}{4.8}{180/7+180}{$\widetilde{L}(\star_1)$}{color = blue}
        \nodeatc{0}{0}{4.7}{360*9/28}{$\widetilde{L}(\star_2)$}{color = blue}

        \nodeatc{12}{0}{4.7}{720/7}{\tiny $\widetilde{LL}(\star_1)$}{color= blue}
        \nodeatc{12}{0}{4.7}{180*9/28+360}{\tiny $\widetilde{LL}(\star_2)$}{color= blue}

        \nodeatc{12}{0}{4.7}{720/7+180}{\tiny $\widetilde{RL}(\star_1)$}{color= blue}
        \nodeatc{12}{0}{4.7}{180*9/28+180}{\tiny $\widetilde{RL}(\star_2)$}{color= blue}

        \nodeatc{12}{0}{4.9}{720/7+90}{\tiny $\widetilde{LR}(\star_1)$}{color= blue}
        \nodeatc{12}{0}{4.9}{180*9/28+90}{\tiny $\widetilde{LR}(\star_2)$}{color= blue}

        \nodeatc{12}{0}{4.9}{720/7+270}{\tiny $\widetilde{RR}(\star_1)$}{color= blue}
        \nodeatc{12}{0}{4.9}{180*9/28+270}{\tiny $\widetilde{RR}(\star_2)$}{color= blue}

    \end{tikzpicture}
    \caption{The generalized cylinder sets for $\alpha=4\pi/7$ associated with words of length not greater than $3$, where $\alpha \in C(LLR)\subseteq C(\star L) \subseteq C(L)$. The dashed lines indicates the boundary of the cylinder sets in Figure~\ref{fig:cylinder}. Note that $C(\star\star L)$ is a union of cylinder sets $C(LLL)$, $C(LRL)$, $C(RRL)$ and $C(RLL)$.}
    \label{fig:generalized_cylinder}
\end{figure}

\begin{lemma}
    Given $x\in \T$ such that $h^{n+1}(x) \neq \alpha$, there exists a unique generalized cylinder set in $\mathcal{GCS}_n$ containing $x$, which we call $GCS_n(x)$.
\end{lemma}

\begin{proof}
    It suffices to check that $x$ is not $\widetilde{u}(\star_i)$ for some $|u| =n$, which is guaranteed by our assumption as $h^{n+1}(\widetilde{u}(\star_i)) = \alpha$ for any $|u|=n$.
\end{proof}

It is also easy to check that generalized cylinder sets behave well in the forward iteration. That is,\begin{equation*}
    h(GCS_n(x)) = GCS_{n-1}(h(x)).
\end{equation*} 
In fact, so do they in the backward iteration as indicated by the following.

\begin{lemma}\label{comp}
    For any $x\in\T$ such that $h^{n+2}(x)\neq \alpha$, we have \begin{equation*}
        GCS^{(1)}_n(x) = \comp^\alpha_x h^{-1}(GCS_n(h(x))) = GCS_{n+1}(x).
    \end{equation*} In consequence, for any $i>0$, if $h^{n+i+1}(x)\neq \alpha$, we have \begin{equation*}
        GCS^{(i)}_n(x) = GCS_{n+i}(x).
    \end{equation*}
\end{lemma}

\begin{proof}
    Note that both $GCS^{(1)}_n(x)$ and $GCS_{n+1}(x)$ are gluing links containing $x$ whose end points are in $\set{\widetilde{u}(\star_i)|~|u|=n+1}$. Such a gluing link is unique.
\end{proof}

Note that given a integer $n$, not all the points $x$ have corresponding $GCS_n(x)$. Therefore, $\set{GCS_n(x)}$ does not form a covering of $\T$.

\subsection{Legal Words and Free Digits}
The following proposition tells us that like cylinder sets, the information of a generalized cylinder set can be encoded by words.

\begin{proposition}[Representation of generalized cylinder sets]\label{rep}
    For each $C \in \mathcal{GCS}_n$, there exists a unique word $u \in \set{L,R,\star}^{n+1}$ (see Figure~\ref{fig:generalized_cylinder}) such that \begin{align*}
        C&= \set{x\in \T: \text{if $u[i]\in\set{L,R}$, then $h^{i-1}(x)$ lies in the semi-circle $C(u[i])$} }\\
        &=\set{x\in \T: \text{if $u[i]\in\set{L,R}$, then $I^\alpha(x)[i] = u[i]$} }.
    \end{align*} In this case, we write $C= C(u)$.
\end{proposition}

\begin{remark}\label{con_legal}
    Given any $u\in\set{L,R,\star}^*$, we can always define \begin{equation*}
    C(u):= \set{x\in \T: \text{if $u[i]\in\set{L,R}$, then $h^{i-1}(x)$ lies in the semi-circle $C(u_i)$} }.
    \end{equation*} When $u$ does not contain any $\star$, $C(u)$ will be the cylinder set for $u$. Thus our notation is consistent.
\end{remark}

\begin{proof}[Proof of Proposition~\ref{rep}]
    We prove this by induction. For $n=0$, $\mathcal{GCS}_0$ are just $\set{C(L),C(R)}$. Now if take any $C\in \mathcal{GCS}_n$ and suppose $h(C)\in \mathcal{GCS}_{n-1}$ is represented by $C(u)$ for some $u\in \set{L,R,\star}^n$. We have the following two cases \begin{enumerate}
	\item If $\alpha \notin C(u)$, then $C$ is either $\widetilde L (C(u))$ or $\widetilde R (C(u))$. We represent $C$ by $Lu$ or $Ru$ respectively.
	\item If $\alpha \in C(u)$, then $C$ is the union of the closures of $\widetilde L (C(u))$ and $\widetilde R (C(u))$. We represent it by $\star u$.
	\end{enumerate}
\end{proof}

\begin{corollary}\label{stars}
    Let $C(u)$ be a generalized cylinder set with $u\in \set{L,R,\star}^n$. Then for any $k\in[n]$, $C(u|_{[k,n]})$ is also a well-defined generalized cylinder set. Moreover, $u[k] = \star$ if and only if $\alpha \in C(u|_{[k+1,n]})$.
\end{corollary}

\begin{definition}[Free digits \& legal words] If $u\in \set{L,R,\star}^*$ can represent some generalized cylinder set, we call $u$ a \textit{legal word} and call all the digits of $u$ at which we have $\star$ the \textit{free digits} of $u$.  
\end{definition}

\begin{lemma}\label{legal}
    Given $g\in \set{L,R}^n$, there exists a unique $u\in\set{L,R,\star}^n$ obtained by replacing some digits of $g$ by $\star$ such that $u$ is legal and $C(u)$ contains $g(\set{\star_1,\star_2})$. We call $u$ the legal word associated with $g$ and write $u = \legal(g)$. In particular, \begin{equation*}
        GCS_n(x) = C(\legal(I^\alpha(x)|_{[1,n+1]})).
    \end{equation*}
\end{lemma}  

\begin{proof}
Note that by the non-crossing property of $\approx_\alpha$, $\widetilde g(\star_1)$ and $\widetilde g(\star_2)$ belongs to the same generalized cylinder set $C\in \mathcal{GSC}_{n-1}$. The legal word $\legal(g)$ associated with $g$ is then just the word representing $C$ from Proposition~\ref{rep}.

We can construct $\legal(g)$ manually from $g$ (See Figure \ref{fig:legal}). Corollary~\ref{stars} implies that $\legal(g)[k] = \star$ if and only if $I^\alpha(\alpha)|_{[1,|g|-k]}$ matches the legal word $\legal(g|_{[k+1,|g|]})$ except at the free digits. We can start by checking from the last digit of $g$ and looking backwards. If the last digit matches the first digit of $I^\alpha(\alpha)$, we replace the second-to-last digit of $g$ with $\star$ to get $\legal(g|_{[|g|-1,|g|]})$. For the $k$-th step, we check whether the last $k$ digits of modified $g$ ($\legal(g|_{[|g|-k,|g|]})$) matches the first $k$ of $I^\alpha(\alpha)$ except at the free digits. If that is the case, we replace the $(n-k)$-th digit of $g$ with $\star$. It follows from Lemma~\ref{stars} that we will in fact end up with a legal word of $g$.
\end{proof}

\begin{figure}[ht]
    \centering
    \begin{tikzpicture}
        \draw[|-|] (0,0) -- (5,0);
        \node at (-0.5,0) {$g:$};
        \draw[->, line width=2pt] (5.2,0) -- (5.7,0) node [right=1pt] {$\legal(g)\in \set{L,R,\star}^{|g|}$};
        \foreach \x in {1,2.5,3.2,4.3}
        {
            \fill [teal](\x-0.2,0) circle (2pt);
            \draw [color=red, line width=1pt](\x,0) to[in angle=180, curve through = {(2.5+0.5*\x, 2.5-0.5*\x)}] (5,3-0.6*\x);
            \draw [->,line width=1pt,teal] (\x-0.2,-0.4) node [teal, below] {$\star$} -- (\x-0.2,-0.12);
        }
        \draw[dashed, red, line width=1pt] (5,0) -- (5,3);
    \end{tikzpicture}
    \caption{Constructing a legal word. Red arcs indicate where $\legal(g|_{[k,n]})$ matches with $I^\alpha(\alpha)|_{[1,n-k+1]}$ up to non-free digits.}
    \label{fig:legal}
\end{figure}
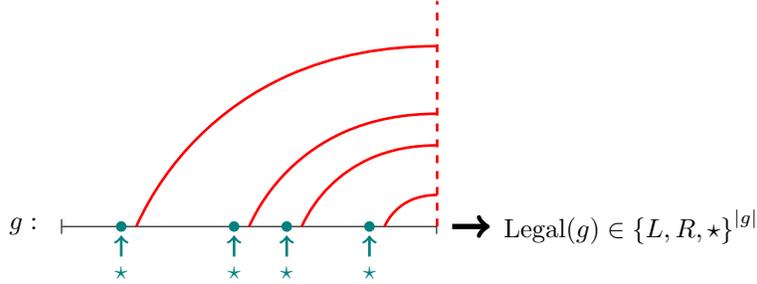

Given $g\in \set{L,R}^n$, if we concatenate $g$ with $h\in \set{L,R}^{k}$, it might happen that $\legal(gh)|_{[1,n]}\neq \legal(g)$. This would be the case only when $\alpha \in C(\legal(h))$ and thus $\legal(gh)[n] = \star$, even though $\legal(g)[n]=g[n]$ will never be $\star$.

\begin{lemma}\label{truncation}
    Given $g\in \set{L,R}^n$ and $h\in\set{L,R}^k$, if $\legal(gh)[n] \neq \star$ or equivalently $\alpha \notin C(\legal(h))$, then $\legal(gh)|_{[1,n]} =  \legal(g)$.
\end{lemma}

\subsection{Cover of $\T$ by Generalized Cylinder Sets}
From now on, we only consider $\alpha$ that are not pre-periodic. That is, none of $h^n(\alpha)$ is periodic. Again we are only excluding a countable set of points. In this case, we have \begin{enumerate}
    \item The itinerary of the angle $I(\alpha):= I^\alpha(\alpha)$ does not have letter $\star$ and thus belongs to $\set{L,R}^\infty$;
    \item For any $x\in\T$, there is at most one $\star$ in $I^\alpha(x)$;
\end{enumerate}

\begin{lemma}[{\cite[Proposition 2.10]{lin2018quasisymmetry}}]
    We denote the infinite word where we replace this $\star$ in $I^\alpha(x)$ with $L$ by $I^L(x)$ and define $I^R(x)$ similarly (If $I^\alpha(x)$ does not contain $\star$, then $I^L(x), I^R(x)$ and $I^\alpha(x)$ are just the same). If $\alpha$ is not preperiodic, we have \begin{equation*}
        [x]_{\approx_\alpha} = \bigcap_{n=1}^\infty C(I^L(x)|_{[1,n]}) = \bigcap_{n=1}^\infty C(I^R(x)|_{[1,n]}).
    \end{equation*}
\end{lemma}

\begin{lemma}[{\cite[Theorem~7]{bandt2006symbolic}}]
    Suppose $\alpha$ is not pre-periodic. Then $[\alpha]_{\approx_\alpha}$ contains at most $2$ points and thus $[\star_i]_{\approx_\alpha}$ has either $2$ or $4$ points. Moreover, in the case when $|[\alpha]_{\approx_\alpha}|=2$ if we denote the other point of $[\alpha]_{\approx_\alpha}$ by $\beta$, then the forward orbit $\set{h^i(\alpha)}_{i=1}^\infty$ does not intersect the arcs $\alpha\frown\beta$, $\widetilde L(\alpha\frown\beta)$ or $\widetilde R(\alpha\frown\beta)$. 
\end{lemma}

In the case when $[\alpha]_{\approx_\alpha}=\set{\alpha,\beta}$, we let $\star_1\frown b_1$ and $\star_2\frown b_2$ denote the two arcs of $\T \backslash [\star_1]_{\approx_\alpha}$ that are not $\widetilde L(\alpha\frown \beta)$ or $\widetilde R(\alpha\frown \beta)$. Let $C_*$ be the gluing link containing $[\star_i]_{\approx_\alpha}$ made of $\overline{\widetilde L(\alpha\frown \beta)}$ and $\overline{\widetilde R (\alpha\frown \beta)}$ (see Figure~\ref{fig:chain_link}). In the case, where $[\alpha]_{\approx_\alpha} = \set{\alpha}$, we define $C_* = \set{\star_1,\star_2}$. Intuitively, we want to think of $C_*$, as well as its images \begin{equation*}
    \set{\widetilde{u}(C_*)|u\in\set{L,R}^*},
\end{equation*} as "bad" sets that we want to stay away from.

\begin{figure}[ht]
    \centering
    \begin{tikzpicture}[scale=0.75]
        \draw [dashed, line width=1pt, red, opacity=0.4] ({5*cos(225)},{5*sin(225)}) -- ({0.5*cos(225)},{0.5*sin(225)});
        \draw [dashed, line width=1pt, red, opacity=0.4] ({0.5*cos(45)},{0.5*sin(45)}) -- ({5*cos(45)},{5*sin(45)});
    
        \fill[blue, opacity=0.4] \arch{0}{0}{5}{225-360}{30} -- \arcc{0}{0}{5}{30}{45} -- \arch{0}{0}{5}{45}{210} --\arcc{0}{0}{5}{210}{225}-- cycle;
        \draw [line width=1pt] (0,0) circle (5);
        \draw [line width=5pt, draw opacity=0.8, teal] ({5*cos(60)},{5*sin(60)}) arc(60:90:5)
        node[pos=0.45,above]{$\alpha\frown\beta$};
        \fill [orange] ({5*cos(60)},{5*sin(60)}) circle (3pt) node [above right=1pt] {$\beta$};
        \fill [red] ({5*cos(90)},{5*sin(90)}) circle (3pt) node [above left=0.6pt] {$\alpha$};
        
        \draw [line width=5pt, draw opacity=0.8, teal] ({5*cos(30)},{5*sin(30)}) arc(30:45:5)
        node[pos=0.3,above right]{$\widetilde R(\alpha\frown\beta)$};
        \fill [orange] ({5*cos(30)},{5*sin(30)}) circle (3pt) node [right=1.2pt] {$\beta/2$};
        \fill [red] ({5*cos(45)},{5*sin(45)}) circle (3pt) node [above=0.6pt] {$\star_1$};
        
        \draw [line width=5pt, draw opacity=0.8, teal] ({5*cos(210)},{5*sin(210)}) arc(210:225:5)
        node[pos=0.3,below left]{$\widetilde L(\alpha\frown\beta)$};
        \fill [orange] ({5*cos(210)},{5*sin(210)}) circle (3pt) node [left=2pt] {$(\beta+1)/2$};
        \fill [red] ({5*cos(225)},{5*sin(225)}) circle (3pt) node [below=1pt] {$\star_2$};

        \node[black] at (0,0) {\Large $C_*$};
    
        \node[red] at (-2,2) {\huge $C(L)$};
        \node[red] at (2,-2) {\huge $C(R)$};
    \end{tikzpicture}
    \caption{When $|[\star_1]_{\approx_\alpha}|=4$, we have the above gluing link $C_*$. In the case when $\beta \in (0,\alpha)$, we have that $\star_1\frown b_1\subseteq C(L)$ and $\star_2\frown b_2 \subseteq C(R)$.}
    \label{fig:chain_link}
\end{figure}
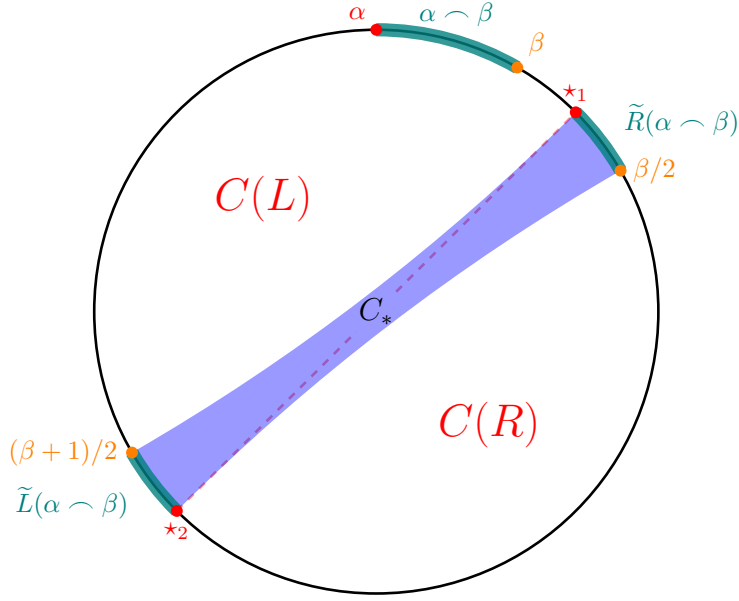

\begin{lemma}\label{well_inside}
    Given an integer $K>0$, there exists $\delta>0$ such that the following holds: for each $x\in T$, there exists at least one generalized cylinder set $C$ in $\mathcal{GCS}_K\cup \mathcal{GCS}_{K+1}$ such that $x\in C$ and for any $\widetilde g(\set{\star_1,\star_2}) \subseteq \p C$, we have \begin{equation*}
	\text{dist}(x, \widetilde g (C_*)) >\delta.
    \end{equation*}
    We denote this gluing link $C$ by $C^K(x)$. 
\end{lemma} 

\begin{proof}
Consider first the case when $[\alpha]_{\approx_\alpha} = \set{\alpha,\beta}$. We observe that since $\alpha \frown \beta$ does not contain any $h^i(\alpha)$, any $\widetilde g$ is well-defined on $\alpha \frown \beta$ and contracts this arc. Thus, if $|\star_1\frown b_2|= |\star_2\frown b_1| = \ell$. Then $\widetilde g (C_*)$ is a gluing link made of two arcs $\overline{\widetilde g(\star_1\frown b_2)}$ and $\overline{\widetilde g(\star_2\frown b_1)}$ with length $\ell/2^{|g|}$. 

By the non-crossing property of $\approx_\alpha$, for any two distinct $g_1,g_2 \in \set{L,R}^*$, either $\widetilde g_1 (C_*)$ and $\widetilde g_2 (C_*)$ are  disjoint or one contains the other. When $-1\leq|g_1| - |g_2| \leq 1$, the containment is impossible and thus $\widetilde g_1 (C_*)$ and $\widetilde g_2 (C_*)$ have to be disjoint. Therefore, given an integer $K>1$, there exists $\delta>0$ such that \begin{equation*}
    \text{dist}(\widetilde g_1 (C_*),\widetilde g_2(C_*)) >2\delta,\quad \text{for any distinct $g_1,g_2$ in $\set{L,R}^{K-1}\cup\set{L,R}^{K}$}.
\end{equation*} Note that $\mathcal{GCS}_K\cup \mathcal{GCS}_{K+1}$ is a finite collection of gluing links whose boundaries are $\widetilde g\set{\star_1,\star_2}$'s with $|g|\in \set{K-1,K}$, the result follows. 
In the case where  $[\alpha]_{\approx_\alpha}=\set{\alpha}$, we can get an analogous result by replacing $C_*$ above with $\set{\star_1,\star_2}$.
\end{proof}

In summary, given an integer $K$, $\set{C^K(x)}$ is a finite closed covering of $\T$ by generalized cylinder sets where each $x$ is uniformly far away from the bad sets $\widetilde{g}(C_*)$ at the boundary of $C^K(x)$. This is the covering of gluing links that we will use to prove Theorem~\ref{main_theo}.

\section{Proof of Theorem~\ref{main_theo}} \label{theo_sec}
\subsection{When $\set{C^K(x)}$ Satisfies Nice Gluing Property}
In this section, we will show that when $d=2$, if $\alpha$ is not strongly recurrent, then each $C^K(x)$ contains a nice gluing circuit around $x$.

\begin{proposition}~\label{good_scale}
    If $\alpha$ is not strongly recurrent as in Definition~\ref{SR}, then for any integer $K$, there exist $(N, C, r)$ depending only on $K$ such that for any $x\in\T$, there is an $(N,C,r)$-nice gluing circuit around $x$ in $C^K(x)$.
\end{proposition}

We will need the following lemma inspired by Lemma~4.2 in \cite{lin2018quasisymmetry}.

\begin{lemma}\label{periodic}
    Suppose that $\alpha$ is not strongly recurrent. Then we have
    \begin{enumerate}
        \item If $|[\star_1]_{\approx_\alpha}| =2$, then we can find two sequence of periodic points $\set{x_n}$ and $\set{y_n}$ in $C(L)$ such that \begin{equation*}
            x_n \approx_\alpha y_n,\quad x_n \to \star_1,\quad y_n\to \star_2.
        \end{equation*}
        \item If $|[\star_1]_{\approx_\alpha}| = 4$, then we can find two sequence of periodic points $\set{x_n}$ and $\set{y_n}$ in $C(L) \backslash C_*$ such that \begin{equation*}
            x_n \approx_\alpha y_n,\quad x_n \to \star_1,\quad y_n\to b_1.
        \end{equation*}
    \end{enumerate}

    Moreover, for each $n$ and $m$, there exist two measures $\mu$ and $\mu'$ of finite logarithmic energy such that \begin{equation*}
        \supp \mu \subseteq x_n\frown x_m,\quad \supp \mu' \subseteq y_n\frown y_m,
    \end{equation*} and there exists $\phi:\supp \mu\to \supp \mu'$ such that $\mu' = \phi \# \mu$ and $x\approx_\alpha \phi(x)$.
\end{lemma}

Note that by symmetry, we can replace $C(L)$ and $b_1$ in the statements with $C(R)$ and $b_2$.We first recall some useful properties of (proper) duplicating intervals.

\begin{lemma}\label{duplicating}
    Given a finite word $g\in \set{L,R}^*$, we have \begin{enumerate}
        \item For a cylinder set $C(g)$, $\partial C(g)$ is the union of all $\widetilde u\set{\star_1,\star_2}$'s where $u$ is an initial word $g|_{[1,i]}$ such that $[i+2,|g|]$ is a duplicating interval. We recall that this means \begin{equation*}
            g|_{[i+2,|g|]} = I(\alpha)|_{[1,|g|-i-1]}.
        \end{equation*} 
        In this case, we call the hyperbolic geodesic $l_u$ a \textit{boundary leaf} of $C(g)$. In particular, $l_{g|_{[1,|g|-1]}}$ is a boundary leaf since $[|g|
        +1,|g|]=\emptyset$ is duplicating. 
        \item Let $\widetilde g$ be the inverse map associated with $g$. Then $\widetilde g$ is well-defined everywhere except exactly at $h^i(\alpha)$ where $[|g|-i+1,|g|]$ is a duplicating interval. In particular, $\widetilde g$ is always ill-defined at $h^0(\alpha)=\alpha$. \end{enumerate}
\end{lemma}

\begin{proof}
    For part~(1), we note that $\widetilde u\set{\star_1,\star_2} \subseteq \partial C(g)$ exactly when \begin{equation*}
        \alpha \in h^{|u|+1}(C(g))=C(g|_{[|u|+1,|g|]}).
    \end{equation*}

    For part~(2), we note that $\widetilde{g}$ is ill-defined at $h^n(\alpha)$ if and only if there exists some $i$ between $2$ and $|g|+1$ such that \begin{equation*}
        \alpha = \widetilde{g|_{[i,|g|]}}(h^n(\alpha)).
    \end{equation*} This only happens when $g|_{[i,|g|]}$ matches exactly with the itinerary of $\alpha$ up to the first $n-1$ steps (here the first step is the original position of $\alpha$).
\end{proof}

\begin{proof}[Proof of Lemma~\ref{periodic}]
    We suppose $|[\star_i]|=4$. The cases for $|[\star_i]|=2$ can be proven using a similar argument and we will omit the details here.

    Since $\alpha$ is not strongly recurrent, there exists integer $D$ such that for any $D$-rare set $\mathcal{R}$, there exist integer $P>2$ and an increasing sequence of integers $\set{n_i}$ with $n_i\leq Pi$ such that $n_i$-th digit of $I(\alpha)$ is not $(D,\mathcal{R})$-duplicating.

    Let $m_i:=n_{3Di}$. Then $\set{m_i}$ is also increasing and each $m_i$ is at least $3D$ digits apart. Since we have $n_i\leq Pi$, it follows that $m_{i}\leq 3PDi$. Thus there exists an increasing sequence of indices $\set{i_k}$ such that \begin{equation*}
	3D \leq m_{i_k} - m_{i_k-1} \leq 3PD
    \end{equation*}By our assumption, for any $i\in\N$, there does not exist any $\mathcal{R}$-duplicating interval with length exceeding $D$ that covers the $m_{i}$-th digit of $I(\alpha)$. In particular, there cannot be any proper duplicating interval with length exceeding $D$ that covers the $m_{i}$-th digit. For any $k\in\N$ and the cylinder set $C(LI^\alpha|_{[1,m_{i_k}]})$, it follows that each boundary leaf $l_{Lu}$ (aside from $\set{\star_1,\star_2}$) satisfies that \begin{equation*}
	m_{i_k} - D \leq |u| \leq m_{i_k}-1.
    \end{equation*} Moreover, since $m_{i_k-1}>m_{i_k}-3PD$, we have $[|u|-3PD,|u|]$ contains $m_{i_k-1}$. Therefore, there cannot be a duplicating interval of $I^\alpha$ covering $[|u|-3PD,|u|]$ (which has length greater than $D$). It follows that if $\widetilde{Lu}$ is ill-defined at some $h^{K}(\alpha)$ with $K>0$, we have \begin{equation*}
	|u|-3PD \leq |u|-K \leq |u|-1.
    \end{equation*} 
    
    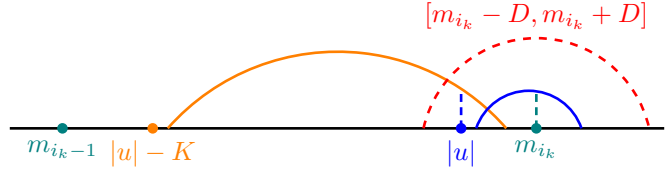
\begin{figure}[ht]
        \centering
        \begin{tikzpicture}
            \draw [line width=1pt] (0,0)--(8.7,0);
            \fill [teal] (0.7,0) circle (2pt) node [below=1pt] {$m_{i_k-1}$};
            \fill [teal] (7,0) circle (2pt) node [below=1pt] {$m_{i_k}$};
            \fill [orange] (1.9,0) circle (2pt) node [below=1pt] {$|u|-K$};
            \fill [blue] (6,0) circle (2pt) node [below=1pt] {$|u|$};
            \draw [orange, line width=1pt](2.1,0) to[curve through = {(4.7,1)}] (6.6,0);
            \draw [blue, line width=1pt](6.2,0) to[curve through = {(6.9,0.5)}] (7.6,0);
            \draw [dashed, red, line width=1pt](5.5,0) to[curve through = {(7,1.2)}] (8.5,0);
            \node[red] (a) at (7,1.5) {$[m_{i_k}-D,m_{i_k}+D]$};
            \draw [densely dashed, blue, line width=1pt] (6,0) -- (6,0.5);
            \draw [densely dashed, teal, line width=1pt] (7,0) -- (7,0.5);
        \end{tikzpicture}
        \caption{Duplicating intervals around $m_{i_k}$. Note that the duplicating interval associated with $u$ covers $m_{i_k}$ (in blue) and thus cannot have length exceeding $D$; and the duplicating interval associated with $|u|-K$ ends in $[m_{i_k}-D,m_{i_k}]$ (in orange) and thus cannot cover $m_{i_k-1}$}.
        \label{fig:intervals}
    \end{figure}

    Therefore $\widetilde{Lu}$ is well-defined outside $\set{h^i(\alpha)}_{i=0}^{3PD}$. Let $\Gamma_i$ denote the components of $C(L) \backslash \set{h^i(\alpha)}_{i=0}^{3PD}$ that contain $\star_i$ respectively. We shrink each $\Gamma_i$ slightly such that $\Gamma_1\cup \Gamma_2$ are positive distance away from $\set{h^i(\alpha)}_{i=0}^{3PD}$ while still contains $b_1$ and let $J$ be the arc $C(L)\backslash \overline{\Gamma_1\cup \Gamma_2}$.

    \begin{figure}[ht]
        \centering
        \begin{tikzpicture}[scale=0.8]
            \draw [line width=1pt] (0,0) circle (5);
            \draw [blue, line width=7pt, opacity=0.5] ({5*cos(170)},{5*sin(170)}) arc(170:222.5:5);
        
            \draw [blue, line width=7pt, opacity=0.5] ({5*cos(42.5)},{5*sin(42.5)}) arc(42.5:65:5);

            \draw [dashed, line width=1pt, red, opacity=0.7] ({5*cos(222.5)},{5*sin(222.5)}) -- ({5*cos(42.5)},{5*sin(42.5)});
        
            \fill [green, fill opacity=0.3] ({5*cos(42.5)},{5*sin(42.5)}) arc(42.5:45:5) -- ({5*cos(45)},{5*sin(45)}) arc({315}:{140}:{5*tan(2.5)}) -- ({5*cos(50)},{5*sin(50)}) arc(50:55:5) -- ({5*cos(55)},{5*sin(55)}) arc(325:270:{5*tan((180-55)/2)}) -- ({5*cos(180)},{5*sin(180)}) arc(180:186:5) -- ({5*cos(186)},{5*sin(186)}) arc(466:286:{5*tan(10)}) -- ({5*cos(206)},{5*sin(206)}) arc(206:222.5:5) -- cycle;

            \nodeatc{0}{0}{5.5}{195}{$\Gamma_2$}{blue}
            \nodeatc{0}{0}{5.5}{55}{$\Gamma_2$}{blue}

            \braceme[line width=2pt, orange, opacity=0.7]{5.4}{68}{167}{br1}{}

            \nodeatc{0}{0}{6.3}{117.5}{$\set{h^i(\alpha)}^{3PD}_{i=0}$}{orange}
            
            \draw [teal, line width=2pt] ({5*cos(55)},{5*sin(55)}) arc(325:270:{5*tan((180-55)/2)}) node[pos=0.5,above=2pt]{\large $l_{g_k}$};

            \nodeatc{0}{0}{2.5}{195}{\large $C(LI^\alpha|_{[1,m_{i_k}]})$}{teal}
            \fill [red] ({5*cos(85)},{5*sin(85)}) circle (3pt) node [below=0.6pt] {$\alpha$};
            
            \fill [red] ({5*cos(42.5)},{5*sin(42.5)}) circle (3pt) node [above right=0.6pt] {$\star_1$};
            
            \fill [orange] ({5*cos(213)},{5*sin(213)}) circle (3pt) node [left=2pt] {$b_1$};
            \fill [red] ({5*cos(222.5)},{5*sin(222.5)}) circle (3pt) node [below left=1pt] {$\star_2$};
            
        \end{tikzpicture}
        \caption{$I_i$, $C(LI^\alpha|_{[1,m_{i_k}]})$ and $l_{g_k}$ in the proof of Lemma~\ref{periodic} when $\star_1\frown b_1\subseteq C(L)$}
    \end{figure}
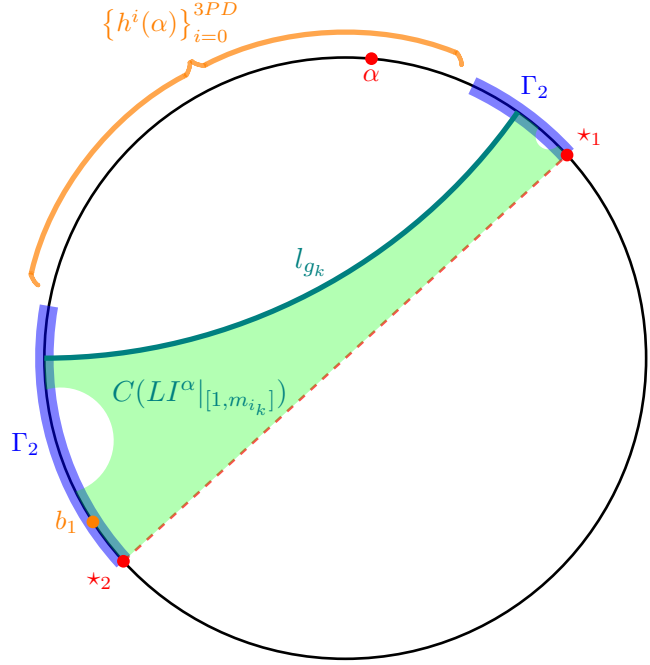
    
    Note that $C(LI^\alpha|_{[1,m_{i_k}]})$ shrinks to $\set{\star_1,\star_2,b_1}$ monotonely as $k\to\infty$. Since we have $\set{h^i(\alpha)}_{i=0}^{3PD}\subseteq\overline{J}\subsetneq \star_1\frown b_1$, we can find $k$ sufficiently large such that \begin{equation*}
	C(LI^\alpha|_{[1,m_{i_k}]}) \cap \overline{J} = \emptyset.
    \end{equation*} For such a cylinder set, there exists a boundary leaf $l_{Lu}$ that connects $\Gamma_i$. We denote $g_k=Lu$ and we have shown that $\widetilde{g_k}$ is well-define outside $\set{h^i(\alpha)}_{i=0}^{3PD}$. In particular, $\widetilde{g_k}$ is a continuous on $\overline{\Gamma_1\cup \Gamma_2}$ and maps $\set{\star_1,\star_2}$ to $\widetilde{g_k}(\set{\star_1,\star_2})$. It follows that $\widetilde{g_k}^2$ is a contraction on each $\overline{\Gamma_i}$ and thus each $\overline{\Gamma_i}$ contains a periodic point with itinerary $\overline{g_k}$. Having the same itinerary, the two periodic points are thus glued by $\approx_\alpha$.

    Note that we have $|g_k|\in [m_{i_k}-D+1,m_{i_k}]$. Thus $|g_k|\to\infty$ as we increase $k$. According to the proof of Lemma~4.2 of \cite{lin2019conformal}, since $\alpha$ is not pre-periodic, for any $k$, there exists a $j>k$ such that $\widetilde{g_k}^2$ and $\widetilde{g_j}^2$ fix different points and thus correspond to different gluing pairs of periodic points. 

    By the proof of Proposition~4.3 in \cite{lin2018quasisymmetry}, between any two gluing pairs of periodic points we constructed above (say associated with $g_k$ and $g_j$ respectively), we can find a pair of generalized Cantor set between the two pairs by the iterative function system $\set{\widetilde{g_k}^2,\widetilde{g_j}^2}$ such that the pair is glued up by a linear transformation. The result then follows.
\end{proof}

\begin{proof}[Proof of Proposition~\ref{good_scale}]
    Given any $C^K(x)\in \set{C^K(y)}_{y\in\T}$, let $\delta>0$ be the constant in Lemma~\ref{well_inside} which indicates the distance between $x$ and the "bad sets" in $C^K(x)$. We will show that there exists an $(N,C,r)$-nice gluing circuit $\set{A_i}\cup \set{A'_i}$ in $C^K(x)$ with the property that for all $A\in \set{A_i}\cup \set{A'_i}$, there exists a boundary leaf $l_g$ such that $\text{dist}(A,\widetilde{g}(C_*)) \leq \delta$. This would imply that $\set{A_i}\cup \set{A'_i}$ is a nice gluing circuit around $x$.

    Suppose that the boundary of our $C^K(x)$ is the union of $\set{\widetilde g_n\set{\star_1,\star_2}}_{n=1}^N$ ordered counterclockwise. Given any boundary leaf $l_{g_n}$ of $C^K(x)$, since $\alpha$ is aperiodic and $\widetilde{g_n}$ is only ill-defined at a subset of $\set{h^i(\alpha)}_{i=0}^{|g_n|}$, $\widetilde{g_n}$ is always defined and continuous at some neighborhood of $C_*$. Then Lemma~\ref{periodic} implies that we can find two sequences of periodic points $\set{x_i}$ and $\set{y_i}$ such that $\widetilde{g_n}(x_i) \approx_\alpha \widetilde{g_n}(y_i)$, and the geodesics connecting $ \widetilde{g_n}(x_i)$ and $ \widetilde{g_n}(y_i)$ approaches $\widetilde{g_n}(C_*)$. In particular, we can find two pairs $\set{\widetilde{g_n}(x_N),\widetilde{g_n}(y_N)}$ and $\set{\widetilde{g_n}(x_M),\widetilde{g_n}(y_M)}$ in this sequence that are less than $\delta$ away from $\widetilde{g_n}(C_*)$. Also by Lemma~\ref{periodic}, we can find two measures $\mu$ and $\mu'$ of finite logarithmic energy supported in $x_N\frown x_M$ and $y_N\frown y_M$ respectively that are glued together via $\approx_\alpha$. We set $A_n = \widetilde{g_n}(x_N\frown x_M)$, $A'_n = \widetilde{g_n}(y_N\frown y_M)$, $B_n = \supp \widetilde{g_n}\#\mu \subseteq A_n$ and $B'_n = \supp \widetilde{g_n}\#\mu' \subseteq A'_n$.
    
    Note that this collection $\set{A_n}\cup \set{A'_n}$ for each $C^K(x)$ is finite and there are also only finitely many different gluing links in $\set{C^K(x)}$. Thus we can find some constants $(N,C,r)$ such that for each $C^K(x)$, $\set{A_n}\cup \set{A'_n}$ is a $(N,C,r)$-nice gluing circuit. By our construction, for each $C^K(x)$, the gluing circuit $\set{A_n}\cup \set{A'_n}$ is less than $\delta$ away from the bad sets $\set{\widetilde{g}(C_*)}$, thus $\set{A_n}\cup \set{A'_n}$ should be around $x$.
\end{proof}

\subsection{When $\set{C^K(x)}$ is Digit-Fixing}
For $\set{C^K(x)}$ to be digit-fixing, we need to introduce a weak pre-periodicity condition.

\begin{definition}[Weak pre-periodicity]
    We say $\alpha\in \T$ is \textit{weakly pre-periodic} with period $k$ with there exists integer $m$ such that $I(\alpha)[m+kn]$ is the same letter for all integer $n\in\N$.
\end{definition}

\begin{lemma}
    Weakly pre-periodic points on $\T$ has measure zero.
\end{lemma}
\begin{proof}
    Given $m,k\in\N$ and $X\in\set{L,R}$, we define \begin{equation*}
	S_{m,k,X}:=\set{u\in \set{L,R}^\infty: u[m+nk] = X\quad\text{for all}\quad n}.
    \end{equation*} Under the metric of the space of sequences, we have that \begin{equation*}
		\dim_{\mathcal{H}}(S_{m,k,X}) \leq \frac{k-1}{k}.
    \end{equation*} According to Proposition~1 of \cite{smirnov2000symbolic} (see also \cite[Section~16]{bruinsymbolic}), pulling back by the map $\varphi: \alpha \mapsto I^\alpha$ does not increase Hausdorff dimension. Thus the result follows from the fact that \begin{equation*}
	\set{\alpha~\text{is weakly pre-periodic}} = \bigcup_{(m,k,X)\in\N^2\times\set{L,R}} \varphi^{-1}(S_{m,k,X}).
    \end{equation*}
\end{proof}

\begin{proposition}\label{digit_matching}
    If $\alpha$ is not strongly recurrent nor is it weakly pre-periodic, then for any integer $L$, there exists $L_\alpha>0$ such that $\set{C^{L_\alpha}(x)}$ is $L$-digit-fixing.
\end{proposition}

Note that if $\alpha \in (C^K)^{(n)}(x)$, then the gluing link $(C^K)^{(n)}(x)$ is either $GCS_{n+K}(\alpha)$ or $GCS_{n+K+1}(\alpha)$ by the property of generalized cylinder sets. Thus by the criterion in Proposition~\ref{digit_fixing}, the proposition above holds due to the following.

\begin{proposition}\label{digit_determining}
    Suppose that $\alpha$ satisfies the same condition as in Theorem~\ref{main_theo}. Then given any integer $L>0$, there exists integer $L_\alpha>0$ such that \begin{equation*}
        GCS_n(\alpha) \cap \set{h^i(\alpha)}_{i=1}^L = \emptyset \quad \text{for all}\quad n\geq L_\alpha.
    \end{equation*}
\end{proposition}

Note that if we replace the generalized cylinder sets $GCS_n(\alpha)$ (which is represented by $\legal(I(\alpha)|_{[1,n+1]})$) with proper cylinder sets $C(I(\alpha)|_{[1,n+1]})$, then it is easy to see that the result is true so long as $\alpha$ is not pre-periodic. I suspect that Proposition~\ref{digit_determining} holds even without the non-weakly pre-periodic condition. In that case, all points on $\T$, except for a set of Hausdorff dimension zero, would satisfy $CCE$. It might be worth looking into.

\begin{conjecture}
    Suppose $\alpha$ is not strongly recurrent nor is it pre-periodic. Then the statement in Proposition~\ref{digit_determining} still holds.
\end{conjecture}

For now to prove the proposition, we need the following lemma.

\begin{lemma}\label{preperiodic}
    Suppose that $\alpha$ is not weakly pre-periodic. Then for any pre-periodic $x\in\T$ with period sufficiently large (depending on $\alpha$), there exists $N_x$ such that \begin{equation*}
	\alpha \notin GCS_n(x),\quad \text{for all}\quad n \geq N_x.
    \end{equation*}
\end{lemma}

The proof of this lemma is done by arguing that the legal word $\legal(I^\alpha(x)|_{[1,n+1]})$ associated with the $GCS_n(x)$ should be approximately pre-periodic. Therefore, if $\alpha$ is in infinitely many $GCS_n(x)$, it should also be approximately pre-periodic. However, without further information, it is hard to control the positions of the free digits for $\legal(I^\alpha(x)|_{[1,n+1]})$, even for pre-periodic $x$. Therefore, instead of dealing with the free digits directly, the proof here uses an algorithm to carefully truncate $\legal(I^\alpha(x)|_{[1,n+1]})$ and we will use some facts about functional graphs to circumvent such difficulty.

\begin{proof}[Proof of Lemma~\ref{preperiodic}]
    Let $k$ be large enough such that $I(\alpha)$ does not start with $k$ same letters. It follows that for any generalized cylinder set $C(u)$, $u$ cannot have $k$ consecutive free digits. We claim that the lemma holds for all pre-periodic $x$ with period greater than $k$. Suppose such an $N_x$ does not exist and we have an increasing sequence of indices $\set{n_l}$ such that \begin{equation*}
        \alpha\in GCS_{n_l-1}(x) = C(\legal(I^\alpha(x)|_{[1,n_l]})).
    \end{equation*}
    Since $x$ is pre-periodic, we can write $I^\alpha(x)$ as $u'\bar u$ for some $u,u'\in \set{L,R}^*$ and $|u|> k$. We let $v^i\in \set{L,R,\star}^{2|u|-i}$ be the legal word associated with $(uu)|_{[1,2|u|-i]}$. We define a map \begin{align*}
	\phi:& [k] \cup \set{0} \to [k]\cup\set{0}\\
	&i \mapsto \text{the number of consecutive $\star$'s right before $v^i[|u|+1]$}.
    \end{align*} The map $\phi$ defines a directed graph $G$ on $[k]\cup \set{0}$ where each vertex has out-degree $1$. Note that by definition, the $(|u|-\phi(i))$-th digit (or $(|u|-i+\phi(i)+1)$-th-to-last digit) of $v^i$ is not free. Now we can associate each $n_l$ with a finite walk $\set{V_k}_{k=0}^{p_l-1}$ in $G$ as follows. We write \begin{equation*}
	I^\alpha(x)|_{[1,n_l]} = u' u^{p_l} r,
    \end{equation*} for some word $r$ with length less than $|u|$ and some integer $p_l$. We set $w_0 = I^\alpha(x)|_{[1,n_l]}$ and $V_0$ be the number of consecutive $\star$'s proceeding the last $r$ digits of $\legal(w_0)$. We will start our walk at $V_0$, remove the last $V_0+|r|$ digits of $I^\alpha(x)|_{[1,n_l]}$ (denoted by $w_1$). Since $(V_0+|r|+1)$-th-to-last digit of $\legal(w_0)$ is not free, by Lemma~\ref{truncation} we have \begin{equation*}
	C(\legal(w_0)) \subseteq C(\legal(w_1)).
    \end{equation*}
    Now for each step, the walker goes from $V_{j-1}$ to $V_j := \phi(V_{j-1})$. We truncate the last $(|u|-V_{j-2})+V_{j-1}$ digits of $w_{j-1}$ to get $w_j$. Note that $w_{j-1}$ ends with $u|_{[1,|u|-V_{j-2}]}$ and thus $\legal(w_{j-1})$ ends with $v^{V_{j-2}}$. Since $(|u|-V_{j-1})$ is not a free digit of $v^{V_{j-2}}$, the $(|u|-V_{j-2}+V_{j-1}+1)$-th-to-last digit of $\legal(w_{j-1})$ cannot be free.  Therefore using Lemma~\ref{truncation} again, we conclude that \begin{equation*}
	C(\legal(w_{j-1})) \subseteq C(\legal(w_j)).
    \end{equation*} After $p_l-1$ steps, we terminate the walk at the vertex $V_{p_l-1}$ and do the last truncation so that \begin{equation*}
	w_{p_l} = u'(u|_{[1,|u|-V_{p_l-1}]}).
    \end{equation*} Note that by our assumption, $\alpha \in  C(\legal(w_0)) \subseteq C(\legal(w_k))$. Therefore, \begin{equation*}
        \legal(I(\alpha)|_{[1,|w_k|]}) = \legal(w_k).
    \end{equation*}

    By our construction, \begin{equation*}
        |w_k| = |u'|+(p_l-k)|u| - V_{k-1}.
    \end{equation*}
    This implies that given our walk $\set{V_k}$, we have \begin{equation*}
        I(\alpha)[|u'| - V_{k-1}+(p_l-k)|u|] = I^\alpha(x)[|u'|- V_{k-1}+(p_l-k)|u|].
    \end{equation*} Now among all the walks, there are infinitely many of them that are trapped in the same cycle $\sigma\subset G$ and end at the same vertex $V$. For those walks, we always have $V_{p_l-1-k |\sigma|} = V$ and thus \begin{equation*}
        |w_{p_l-k|\sigma|}| = |u'| - V_{p_l-1-k|\sigma|}+(p_l-(p_l-1-k|\sigma|+1))|u| = |u'| - V + k|\sigma||u|
    \end{equation*}
    
    It follows that for any $k\geq 1$, the $(|u'|-|V|+ k|\sigma||u|)$-th digits of $I^\alpha(x)$ and $I(\alpha)$ coincide. Since they all have to be the same letter by the pre-periodicity of $x$, this contradicts the fact that $\alpha$ is not weakly pre-periodic.
\end{proof}

\begin{proof}[Proof of Proposition~\ref{digit_determining}]
    Given any $i\in[L]$, we can choose an integer $n_i$ large enough such that $h^i(\alpha)$ is not in the cylinder set $ C(I^\alpha|_{[1,n_i]})$. Since $\alpha \in C(I^\alpha|_{[1,n_i]})$, we can select a boundary pair $\widetilde u(\set{\star_1,\star_2})$ of the cylinder set $C(I^\alpha|_{[1,n_i]})$ that $l_u$ separates $h^i(\alpha)$ and $\alpha$. Let $k$ be the integer in Lemma~\ref{periodic}. Note that there exist only finitely many periodic points in $\T$ with period not greater than $k$.
    
    We first suppose that $|[\star_i]_{\approx_\alpha}| = 2$. Note that $\widetilde u$ will be continuous near $\set{\star_1,\star_2}$. By Lemma~\ref{periodic}, we can find a gluing pair $\set{x_1,x_2}$ of periodic points with itinerary $\overline{g}$ arbitrarily close to $\set{\star_1,\star_2}$ with period larger than $k$. For such a pair, $\widetilde u(\set{x_1,x_2})$ will be a pair of pre-periodic points with itinerary $u\overline{g}$ that is close to $\widetilde {u} (\set{\star_1,\star_2})$ and thus the geodesic connecting $\widetilde u(\set{x_1,x_2})$ also separates $h^i(\alpha)$ and $\alpha$. 

    If $|[\star_i]_{\approx_\alpha}| = 4$ instead. The function $\widetilde u$ will be continuous in a neighborhood of $C_*$ on $\T$. Since the lamination $\approx_\alpha$ satisfies the non-crossing property, $\widetilde u(C_*)$ will separate $h^i(\alpha)$ and $\alpha$. It follows that one of $\widetilde u(\set{\star_i,b_i})$ will also separate the two. For this $i$, by Lemma~\ref{periodic}, we can find a gluing pair $\set{x_1,x_2}$ of periodic points with itinerary $\overline{g}$ arbitrarily closed to $\set{\star_i,b_i}$ larger than $k$. Hence, $\widetilde u(\set{x_1,x_2})$ will be a pair of pre-periodic points with itinerary $u\overline{g}$ that is close to $\widetilde {u} (\set{\star_i,b_i})$ and thus the geodesic connecting $\widetilde u(\set{x_1,x_2})$ separates $h^i(\alpha)$ and $\alpha$. 
    
    By Lemma~\ref{preperiodic}, there exists $L_i$ such that \begin{equation*}
	\alpha \notin GCS_n(\widetilde u (x_1))= GCS_n(\widetilde u (x_2)) \quad \text{for all}\quad n\geq L_i.
    \end{equation*} Thus there is a generalized cylinder set of any degree greater than $L_i$ that is between $h^i(\alpha)$ and $\alpha$. It follows that \begin{equation*}
		h^i(\alpha) \notin GCS_n(\alpha) \quad \text{for all}\quad n\geq N_i.
    \end{equation*} The result follows by choosing $L_\alpha = \max_{i\in[L]}L_i$.	
\end{proof}

\subsection{Proving Theorem~\ref{main_theo}}
The proof of Theorem~\ref{main_theo} is outlined in Figure~\ref{fig:flow_chart}.
\begin{proof}[Proof of Theorem~\ref{main_theo}]
    Since $\alpha$ is not strongly recurrent, by Theorem~\ref{smirn}, there exist $(L,P,M)$ such that $\alpha$ satisfies semi-CCE condition for any $L$-digit-fixing covering $\set{D(x)}_{x\in\T}$ with respect to parameters $(L,P,M)$. Since we additionally assume that $\alpha$ is not weakly pre-periodic, by Proposition~\ref{digit_matching}, there exists $L_\alpha$ such that $\set{C^{L_\alpha}(x)}$ is a $L$-digit-fixing covering of gluing links. By Proposition~\ref{good_scale}, there exist $(N,C,r)$ such that each $C^{L_\alpha}(x)$ contains an $(N,C,r)$-nice chain around $x$. Thus $\alpha$ satisfies CCE condition with the covering $\set{C^{L_\alpha}(x)}_{x\in\T}$ and parameters $(N,C,r,P,M)$. The result then follows.
\end{proof}

\bibliographystyle{abbrv}
\bibliography{main}

\end{document}